\documentclass[11pt]{amsart}
\usepackage{amssymb, amsmath, amsthm,amsrefs}
\usepackage{latexsym}
\usepackage{color}
\usepackage{exscale}
\usepackage{fullpage}
\usepackage{rotating}
\usepackage{bm, bbm, mathrsfs}
\usepackage{graphics, graphicx}
\newtheorem{theorem}{Theorem}[section]
\newtheorem{lemma}[theorem]{Lemma}
\newtheorem{proposition}[theorem]{Proposition}

\newtheorem{definition}{Definition}

\newtheorem{example}[theorem]{Example}

\newtheorem{remark}{Remark}[section]

\def\vp{\varphi}

\def\eq#1{(\ref{#1})}
\def\nn{\nonumber}
\def\({\left(\begin{array}{cccccc}}
\def\){\end{array}\right)}

\def\eq#1{(\ref{#1})}
\def\nn{\nonumber}

\def\({\left(\begin{array}{cccccc}}
\def\){\end{array}\right)}

\def\bes{\begin{eqnarray}}
\def\ees{\end{eqnarray}}

\newcommand{\beq}{\begin{equation}}
\newcommand{\eeq}{\end{equation}}
\newcommand{\bea}{\begin{eqnarray}}
\newcommand{\eea}{\end{eqnarray}}
\newcommand{\beann}{\begin{eqnarray*}}
\newcommand{\eeann}{\end{eqnarray*}}

\newcommand{\lam}{\ensuremath{\lambda}}

\newcommand{\RR}{\mathbb{R}}

\newcommand{\br}{\bar r}

\newcommand{\pf}{\begin{proof}}
\newcommand{\foorp}{\end{proof}}
\DeclareMathOperator{\grad}{grad}
\DeclareMathOperator{\dv}{div}

\DeclareMathOperator{\supp}{supp}

\numberwithin{equation}{section}
 
\begin{document}

\title{On similarity flows for the compressible Euler system}

\begin{abstract}
Radial similarity flow offers a rare instance where concrete inviscid, multi-dimensional, 
compressible flows can be studied in detail. In particular, there are 
flows of this type that exhibit imploding shocks and cavities. In such flows 
the primary flow variables (density, velocity, pressure, temperature) become unbounded
at time of collapse. 
In both cases the solution can be propagated beyond collapse by having an expanding 
shock wave reflect off the center of motion.

These types of flows are of relevance in bomb-making and inertial confinement fusion, 
and also as benchmarks for computational codes; they have been investigated 
extensively in the applied literature. 
However, despite their obvious theoretical interest as examples of unbounded 
solutions to the multi-dimensional Euler system, the existing literature does 
not address to what extent such solutions are {\em bona fide} weak solutions.

In this work we review the construction of globally defined radial similarity shock 
and cavity flows, and give a detailed description of their behavior following collapse. 
We then prove that similarity shock solutions 
provide genuine weak solutions, of unbounded amplitude, to the multi-dimensional 
Euler system. However, both types of similarity flows involve regions of 
vanishing pressure prior to collapse (due to vanishing temperature
and vacuum, respectively) - raising the possibility that Euler flows may
remain bounded in the absence of such regions.
\end{abstract}

\author{Helge Kristian Jenssen}\address{H.\ K.\ Jenssen, Department of
Mathematics, Penn State University,
University Park, State College, PA 16802, USA ({\tt
jenssen@math.psu.edu}).}

\author{Charis Tsikkou}\address{C. Tsikkou, Department of
Mathematics, West Virginia University,
Morgantown, WV 26506, USA ({\tt
tsikkou@math.wvu.edu}).}

\date{\today}
\maketitle

\tableofcontents

%%%%%%%%%%%%%%%%%%%%%%%%%%%%%%%%%%%%%%%%%%%%%%%%
\section{Introduction}\label{intro}
%%%%%%%%%%%%%%%%%%%%%%%%%%%%%%%%%%%%%%%%%%%%%%%%
We consider two types of radial similarity flows 
for the compressible Euler system. These are particular types of
solutions with planar (slab), cylindrical, or spherical symmetry.\footnote{While all three types of flows are ``one-dimensional"
in  the sense that they depend on a single spatial variable $r$, we reserve this term for 
the  case of slab symmetry (i.e., the case when there is a fixed direction in physical 
space such that, at each fixed time, all flow quantities are constant in planes normal to this direction).} 
Under a similarity assumption the Euler system reduces to a coupled, 
nonlinear system of ODEs with respect to a similarity 
variable $x=t/r^\lambda$, where $t$ is time, $r$ is distance to the origin, and $\lambda$ is the 
similarity exponent.
 Similarity flows provide a rare instance where exact solutions
to the multi-dimensional compressible Euler system can be constructed ``by hand''
and studied in considerable detail. Following Guderley's pioneering 
study \cite{gud}, they have attracted substantial attention from physicists, 
engineers, and mathematicians. For a recent overview of the literature, see 
\cite{rkb_12} and references therein.

The existing literature provides examples of similarity flows 
where a single (spherical or cylindrical) incoming shock wave propagates into a quiescent region
about the origin (i.e., the fluid there is at rest and at constant density 
and pressure). The shock strengthens as it approaches the origin and the shock speed becomes 
unbounded at the instance of collapse at the 
origin. (For convenience, the time of collapse is chosen as $t=0$.) 
One can construct a complete (similarity) solution for all later times as well by 
having a diverging shock wave reflect off the origin. A different type of
similarity solution describes the situation where a gas fills a spherical or cylindrical 
cavity (vacuum region) near the origin. Again, the speed of the fluid-vacuum interface blows 
up at collapse. Also in this case  
a global-in-time similarity solution can be constructed by inserting an 
outgoing shock after collapse. We refer to these two types of solutions 
as {\em similarity shock} and {\em similarity cavity flows}, respectively.

In either case the profiles for the fluid velocity, pressure, sound speed, and 
temperature at time of collapse are unbounded, with behavior given by negative 
powers of $r$ (in the cavity case, this applies also to the density profile).  
For this the similarity exponent must satisfy $\lambda>1$.

In Section \ref{eqns} we record the multi-dimensional (multi-d) Euler equations for compressible 
flow of an ideal and polytropic gas with adiabatic exponent $\gamma>1$, including its 
radial form. We also posit the form of the radial similarity solutions under consideration. 
Section \ref{Gud_soln} outlines the setup for each type of solutions and collects various properties 
(initial data, jump relations, characteristics, etc.) of the similarity solutions under consideration.

For the actual construction of physically relevant similarity solutions with these properties, 
we follow Lazarus \cite{L} who treats both shock and cavity flows. A complete breakdown 
of the various possibilities, including the key determination of allowed values of the 
similarity exponent $\lambda$, requires a detailed analysis and numerical calculations. 
Our main purpose of verifying that the Euler system admits unbounded weak solutions,
does not require a full breakdown of all the cases.
Instead, Section \ref{constr_sim_solns} outlines enough of this analysis to obtain {\em some}
cases of Euler flows with unbounded amplitudes. In particular, we restrict attention to the 
standard value of the similarity exponent $\lambda$. This is the so-called 
``analytic'' value, denoted $\lambda_{std}$ by Lazarus \cite{L}. See Section 
\ref{constr_sim_solns} for details, where we also describe how the solutions are 
propagated past collapse to yield complete (i.e., global-in-time), radial similarity flows.

The resulting, well-known, solutions can be studied in detail.
In particular, we deduce their asymptotic behavior at $x=+\infty$, which plays a key role 
in the analysis that follows. It turns out that the behavior of the resulting flows
after collapse is markedly different near the center of motion in the shock case 
and in the cavity case; see Section \ref{refl_shck}.
We also include a discussion to the effect that, at least among similarity flows,
the continuation beyond collapse appears to be uniquely determined for both 
types of flows.  
Note that all jump discontinuities appearing in these similarity flows are, 
by construction, entropy admissible: both the incoming and the reflected shocks are 
compressive.

We then turn to our main concern: to what extent these types of similarity flows 
represent genuine weak solutions of the original, multi-d compressible Euler system. 
As the similarity solutions are singular and suffer blowup of primary flow variables at 
the origin, it is not immediately clear in what sense the weak form is satisfied.
While some authors \cites{bk,L} have addressed the constraint of locally finite energy 
for the similarity flows under consideration, we are not aware of a complete 
analysis. Concentrating on similarity shock solutions, we demonstrate that the
flows constructed in the literature
are indeed {\it bona fide} weak solutions whenever the similarity exponent $\lambda$
satisfies the constraint $\lambda\leq\frac{n}{2}+1$, where $n$ is $2$ or $3$ for
cylindrical or spherical flow, respectively. The numerical values available in the 
literature indicate that the solutions corresponding to the particular value 
$\lambda_{std}$ always satisfy this constraint.

We shall show that the similarity shock solutions under consideration are {\it bona fide} weak 
solutions in the following sense: all terms occurring in the weak formulation of
the multi-d Euler system are locally integrable in space-time;  the amounts of mass, 
momentum, and energy within any fixed, compact spatial region change continuously 
with time (in particular, they are finite); and finally, the weak forms of the equations are satisfied.
(Their total mass, momentum and energy in all of space are 
not bounded; however, this could be arranged via suitable modifications away from the origin
without affecting the blowup behavior near the origin.)

We emphasize that we verify the weak form of the original, {\em multi-d} Euler system. 
Since the similarity solutions under consideration are radially symmetric, it is 
convenient to first derive the corresponding weak formulation for general radial solutions.
This requires some care as the latter formulation involves different types of 
``test functions'' for the different conservation laws. For completeness we include
the derivation of the radial weak form of the equations (see Definition 
\ref{rad_symm_weak_soln} and Proposition \ref{rad_md}; here we follow the 
analysis \cite{hoff} for radial  Navier-Stokes flow).

With these preparations, Section \ref{sim_weak_solns} provides the details 
of the proof that genuine multi-d weak solutions are obtained from the radial 
symmetry solutions.

%%%%%%%%%%%%%%%%%%%%%%%%%%%%%%%%%%%%%%
\paragraph{\bf Discsussion}
%%%%%%%%%%%%%%%%%%%%%%%%%%%%%%%%%%%%%%
The existence of singular flows suffering point-wise blowup of flow variables
is of obvious relevance in connection with the general Cauchy problem for the 
compressible Euler system. With the notable exception of small-variation data 
near a strictly hyperbolic state (Glimm \cite{glimm}), there is currently no general, 
global-in-time existence result available for the one-dimensional (1-d) Cauchy problem 
for hyperbolic systems. (See \cites{liu77,temple81} for extensions that cover 
certain types of large variation data specifically for the Euler system.)
In more than one space dimension the situation is bleaker, and symmetric 
flows offer a natural case to consider in isolation. 
For results on isothermal and isentropic radial flow with general data, 
see \cites{cp,cs,mmu}.

In view of the blowup exhibited by similarity shock and similarity cavity 
solutions, it would appear that any existence result, applying to ``general'' data, 
for the multi-d Euler system would necessarily have to involve 
unbounded solutions. However, one should be careful not to draw
too general conclusions on the basis of the similarity flows we study 
here. These are exceedingly special solutions, some aspects of which are 
borderline physical.
In particular, both types of flows involve regions of vanishing pressure prior 
to collapse.
In the case of a collapsing cavity this is due to the vacuum, and there is no reason why
the Euler model should provide an accurate description close to its collapse.
For the converging shock case, it turns out that the quiescent state into which the 
converging shock propagates, must necessarily be at zero pressure (due to vanishing 
temperature there) in order to generate an exact solution. In approximate 
treatments this amounts to a ``strong shock'' assumption.

For the case of an incoming shock, it is physically reasonable that a nonzero 
counter pressure would slow it down and possibly prevent unbounded 
amplitudes.
This would provide a mechanism to ``save'' the Euler model from actual 
blowup.\footnote{The situation for radial {\em isentropic} similarity flow 
(constant entropy throughout, disregarding the energy equation
\cite{daf}) does not contradict this picture. In that case
a converging similarity shock can propagate into  quiescent region 
only if $\lambda=1$;  no blowup of primary flow variables occurs, and the upstream pressure 
is strictly positive. The same applies to radial isothermal similarity flow.}
In particular, if indeed correct, this would 
show that the strong shock approximation fails to capture a crucial aspect of 
exact solutions near collapse of symmetric shock waves (blowup vs.\ no blowup of 
primary flow variables). We note that a number of works consider the effect of a positive 
counter pressure, e.g.\ \cites{ah,phpm,welsh,vrt} and 
references therein. However, while amplitude blowup is still present in these works, 
none of them provide {\em exact} weak solutions to the Euler system.

The conventional point of view appears to be that the blowup 
exhibited by radial similarity flows results from multi-d wave focusing, 
much like what occurs for radial solutions to the linear 3-d wave equation.
The remarks above raise the possibility that the unbounded amplitudes could 
be due to the presence of regions of vanishing pressure.
We are not aware of a definite argument one way or the other - 
possibly both effects are required to generate blowup in $L^\infty$.
Unfortunately, 1-d (slab symmetry) similarity flows do not help in assessing 
the situation: such solutions fail to generate physically 
acceptable flows; see Remark \ref{no_1_d_ex}.

%%%%%%%%%%%%%%%%%%%%%%%%%%%%%%%%%%%%%%%%%%%%%%%%
\section{Equations}\label{eqns}
%%%%%%%%%%%%%%%%%%%%%%%%%%%%%%%%%%%%%%%%%%%%%%%%
The full, multi-d Euler system for compressible gas flow is given by
\begin{align}
	\rho_t+\dv\left(\rho \vec u\right)&= 0\label{m_d_mass}\\
	\left(\rho \vec u \right)_t+\dv\left(\rho \vec u\otimes \vec u\right)+\grad p &= 0\label{m_d_mom}\\
	\Big[\rho e+\frac{\rho|\vec u|^2}{2}\Big]_t
	+\dv\Big[\Big(\rho e+\frac{\rho |\vec u|^2}{2}+p\Big)\vec u\Big]&= 0.\label{m_d_energy}
\end{align}
The variables are $\rho=$ density, $\vec u=$ fluid velocity, $p=$ pressure, $e=$ specific 
internal energy. 
Under the assumption of radial symmetry (i.e., all unknowns depend only on time $t$ and 
radial distance $r$ to the origin or an axis of symmetry, and 
$\vec u$ is purely radial), the system takes the form: ($u=|\vec u|$)
\begin{align}
	\left(r^m\rho\right)_t+\left(r^m\rho u\right)_r &= 0\label{mass}\\
	\left(r^m\rho u \right)_t+\left(r^m(\rho u^2+p)\right)_r &= mr^{m-1}p\label{mom}\\
	\Big(r^m\rho \Big[e+\frac{u^2}{2}\Big] \Big)_t
	+\Big(r^m\rho u\Big[e+\frac{u^2}{2}+\frac{p}{\rho}\Big]\Big)_r &= 0.\label{energy}
\end{align}
Here $r$ varies over $\RR^+$, subscripts denote differentiation, and $m=1,\, 2$ for flows with 
cylindrical or spherical symmetry, respectively. With $m=0$ and $r$ varying over $\RR$,
we have the one-dimensional Euler system.  
We consider an ideal, polytropic gas with equation of state
\beq\label{perf}
	p=(\gamma-1)\rho e=(\gamma-1)c_v\rho \theta,
\eeq
where $\gamma>1$ and $c_v$ are positive constants, and $\theta=$ temperature. 
The specific entropy $S$ is related to $p$ and $\rho$ by
\beq\label{entr}
	p\rho^{-\gamma}=\text{Constant}\cdot \exp\Big(\frac{S}{c_v}\Big).
\eeq
It is a consequence of the conservation laws above that $S$ remains constant along 
particle trajectories in smooth regions of the flow:
\beq\label{entrpy_eul}
	S_t+uS_r = 0.
\eeq
The sound speed $c$ is given by
\beq\label{sound}
	c^2:=\frac{\gamma p}{\rho}\equiv \gamma(\gamma-1)e,
\eeq
and with $u$, $c$, and $\rho$ as primary unknowns, the system 
\eq{mass}-\eq{energy} takes the form:
\begin{align}
	u_t+ uu_r +\frac{1}{\gamma\rho}(\rho c^2)_r&= 0\label{u}\\ 
	c_t+uc_r+\frac{(\gamma-1)}{2}c\Big(u_r+\frac{mu}{r}\Big)&=0\label{c}\\
	\rho_t+u\rho_r+\rho\Big(u_r+\frac{mu}{r}\Big) &= 0.\label{rho}
\end{align}
Following the notation and setup of Lazarus \cite{L}, we introduce the similarity 
coordinate
\beq\label{sim_coord}
	x=\frac{t}{r^\lam},
\eeq
where $\lambda$ is the similarity exponent (to be determined), and make the ansatz
\begin{align}
	u(t,r) &= -\frac{r}{\lam t}\ V(x)=-\frac{r^{1-\lam}}{\lam}\frac{V(x)}{x} \label{V}\\
	c(t,r) &= -\frac{r}{\lam t}\ C(x)=-\frac{r^{1-\lam}}{\lam}\frac{C(x)}{x} \label{C}\\
	\rho(t,r) &=r^\kappa R(x),\label{R}
\end{align}
where $\kappa$ is a constant.
We refer to solutions with this particular structure as {\em similarity flows}. 
Their relevance relies on the fact that they include physically 
meaningful flows where either symmetric shocks or cavities implode 
(converge, focus, collapse) at the origin. Similarity flows are determined via solutions to 
ODEs for $V$, $C$, and $R$. These are the {\em similarity ODEs} which we record 
in Section \ref{sim_ODEs} below. We stress that, differently from many other 
cases of similarity solutions, the similarity exponent $\lambda$ is not given a priori,
but must be determined as part of the solution.

\section{Similarity shock and similarity cavity solutions}\label{Gud_soln}
%%%%%%%%%%%%%%%%%%%%%%%%%%%%%%%%
\subsection{Similarity shock solutions}\label{sim_shocks}
%%%%%%%%%%%%%%%%%%%%%%%%%%%%%%%%
We shall first consider similarity flows where a single (spherical, cylindrical, 
or planar) shock moves toward the origin for negative times, and 
focuses at the origin at time $t=0$. Taking the existence 
of such similarity flows for granted for now, in this section we consider the 
Rankine-Hugoniot conditions, describe various constraints 
that should be met by physically relevant similarity flows, and describe
a particular (critical) characteristic which plays a central role in 
the construction of such flows.

First, the flows on both sides of the shock are assumed to be similarity  flows
with the same values of $\lam$, $\gamma$, and $\kappa$ in \eq{V}-\eq{R}.
We assume that the converging shock path is described by a constant 
value of the similarity variable $x$, say
\beq\label{path}
	x\equiv -1 \qquad\text{so that}\qquad r_{shock}=(-t)^\frac{1}{\lam},\quad t<0.
\eeq
We shall only consider situations where the shock reaches the origin 
with infinite speed, so that
\beq\label{constr_1}
	\lam>1.
\eeq
We follow \cite{L} and let subscripts $0$ and $1$ denote 
evaluation immediately ahead of and behind of the shock, 
respectively. The (exact) jump relations and entropy condition 
then take the forms
\begin{align}
	1+V_1 &=\frac{\gamma-1}{\gamma+1}(1+V_0)
	+\frac{2C_0^2}{(\gamma+1)(1+V_0)}\label{V_jump}\\
	C_1^2 &= C_0^2+\frac{\gamma-1}{2}[(1+V_0)^2-(1+V_1)^2] \label{C_jump}\\
	R_1(1+V_1) &=R_0(1+V_0)\label{R_jump}\\
	C_0^2 &< (1+V_0)^2.\label{entropic}
\end{align}
Here \eq{entropic} expresses that the shock is supersonic relative to the 
state ahead; together these imply $C_1^2 > (1+V_1)^2$, amounting to
the admissibility of the similarity shocks.
The fluid on the inside (ahead) of the converging shock is assumed to be at rest and 
at constant density and pressure (quiescent state). According to \eq{R}, 
the constant density there dictates that $\kappa=0$ and $R(x)$ is constant;
for concreteness let
\[R(x)\equiv 1\qquad\text{for $-\infty<x<-1$.}\] 
Next, for an ideal gas $c^2\propto \frac{p}{\rho}$, so that 
the sound speed is constant in the quiescent region. 
As we assume $\lam\neq 1$, \eq{C} implies that $C$ must
vanishes identically there. As the fluid near the origin is assumed to 
be at rest, we therefore have
\[V(x)=C(x)\equiv 0 \qquad \text{for $-\infty<x<-1$.}\] 
We are thus considering a 
single, converging shock which moves into a quiescent region at 
zero pressure and unit density. For an ideal polytropic gas, this 
means that the temperature vanishes identically in the region 
inside the converging shock.

With $(V_0,C_0,R_0)=(0,0,1)$, inequality \eq{entropic} is satisfied, and the jump 
relations \eq{V_jump}-\eq{R_jump} give the following 
initial conditions for the similarity variables $V$, $C$, $R$ at $x=-1^+$:
\beq\label{init_data}
	V(-1)=V_1=-\frac{2}{\gamma+1},\qquad 
	C(-1)=C_1=\frac{\sqrt{2\gamma(\gamma-1)}}{\gamma+1},\qquad 
	R(-1)=R_1=\frac{\gamma+1}{\gamma-1}.
\eeq
Along the immediate outside of the converging shock, the primary 
flow variables are therefore given by \eq{V}-\eq{R} as (recall that
$\kappa=0$ in the present shock case):
\beq\label{u_c_rho_at_skock}
	u=\frac{V_1}{\lam}r^{1-\lam}\qquad c
	=\frac{C_1}{\lam}r^{1-\lam}\qquad \rho\equiv R_1. 
\eeq
As we assume $\lam>1$, it follows that the velocity $u$ and 
the sound speed $c$ become unbounded along the outside 
of the shock as it collapses at the origin, while the density 
remains finite. (The same applies along any curve given by 
$x\equiv constant \in(-1,0)$.)

Next, we are only interested in solutions where the 
flow variables $u$, $c$ and $\rho$ are ``well behaved'' at 
any location $r>0$ at time $t=0$. 
In particular, for any fixed $r>0$ we require that 
$u(t,r)$ and $c(t,r)$ tend to finite limits as $t\to 0$, 
i.e., as $x\to 0$.
According to \eq{V} and \eq{C} we must therefore have that 
\beq\label{V/x_C/x-zero}
	\ell:=\lim_{x\to 0} \frac{V(x)}{x}\qquad \text{and}\qquad 
	L:=\lim_{x\to 0} \frac{C(x)}{x}\qquad\text{are finite,}
\eeq
%We will therefore require that the functions 
%\beq\label{cont_at_0}
%	\frac{V(x)}{x},\, \frac{C(x)}{x}\qquad\text{are continuous at $x=0$.}
%\eeq
Thus, in particular, we have
\beq\label{vc_zero}
	V(0)=C(0)=0.
\eeq
It then follows from \eq{V}-\eq{C} and \eq{V/x_C/x-zero} that, at time of collapse 
($t=0$), the radial flow speed $u$ and the sound speed $c$ blow up 
according to 
\beq\label{uc_0}
	u(0,r)=-\frac{\ell}{\lambda}r^{1-\lam}\qquad\text{and}\qquad c(0,r)=-\frac{L}{\lambda}r^{1-\lam},
\eeq
while the density is constant, $\rho(0,r)\equiv R(0)$. 
As a consequence, the pressure and temperature profiles at time of collapse 
blow up according to
\beq\label{ptheta_0}
	p(0,r),\, \theta(0,r)\, \propto \, r^{2(1-\lam)}.
\eeq

We point out that the limits in \eq{V/x_C/x-zero} will turn out to be non-zero and finite 
for the solutions constructed below. It follows that all three characteristic speeds
($u\pm c$ and $u$) are bounded at all points except at $(t,r)=(0,0)$.
In particular, all fluid particles, except the one at the origin, are located away from 
$r=0$ at time $t=0$; in other words, the solutions under consideration are not 
of ``cumulative'' type where all (or a part of) the mass concentrates at the origin at collapse 
(examples of such flows are given in \cites{kell,am}).

Next we note that,  by \eq{init_data}, 
\beq\label{init_over}
	C>1+V>0\qquad\text{at $x=-1$},
\eeq
while \eq{vc_zero} shows that the opposite inequality holds at $x=0$. 
Thus, for some critical $x_c\in (-1,0)$ we must have 
\[1+V(x_c)=C(x_c).\]
(For the solutions considered below, there is a unique critical value $x_c$.)
Now, to determine the full solution of the flow problem before collapse, 
we must integrate the similarity ODEs 
for $V(x)$, $C(x)$, and $R(x)$ for $x\in (-1,0)$, subject to the initial data in 
\eq{init_data}. It so happens that these ODEs are singular at points where $1+V=C$
(see \eq{V_sim2}-\eq{D}), and we have 
just seen that this must occur at some point $x_c\in (-1,0)$. 
The corresponding curve in the $(t,r)$-plane turns out to be a 1-characteristic for the 
corresponding Euler flow. (More generally, a calculation shows that the curve 
$x\equiv \bar x=constant$ is a 1-characteristic if and only if $1+V(\bar x)=C(\bar x)$.)
%along 
%\[r_c(t):=\big(\textstyle\frac{t}{x_c}\big)^\frac{1}{\lam}\]
%we have
%\[\dot r_c=\frac{r_c^{1-\lam}}{\lam x_c}\equiv
%\frac{r_c^{1-\lam}}{\lam x_c}\left(C(x_c)-V(x_c)\right)=u(t,r_c(t))-c(t,r_c(t)).\]
Passing through $x=x_c$ corresponds to crossing the {\em critical
1-characteristic}, i.e. the 1-characteristic that catches up with the converging shock 
as it collapses at the origin. See Figure 1.

%%%%%%%%%%%%%%%%%%%
%	FIGURE 
%%%%%%%%%%%%%%%%%%%
\begin{figure}\label{Figure_1}
	\centering
	\includegraphics[width=9cm,height=8cm]{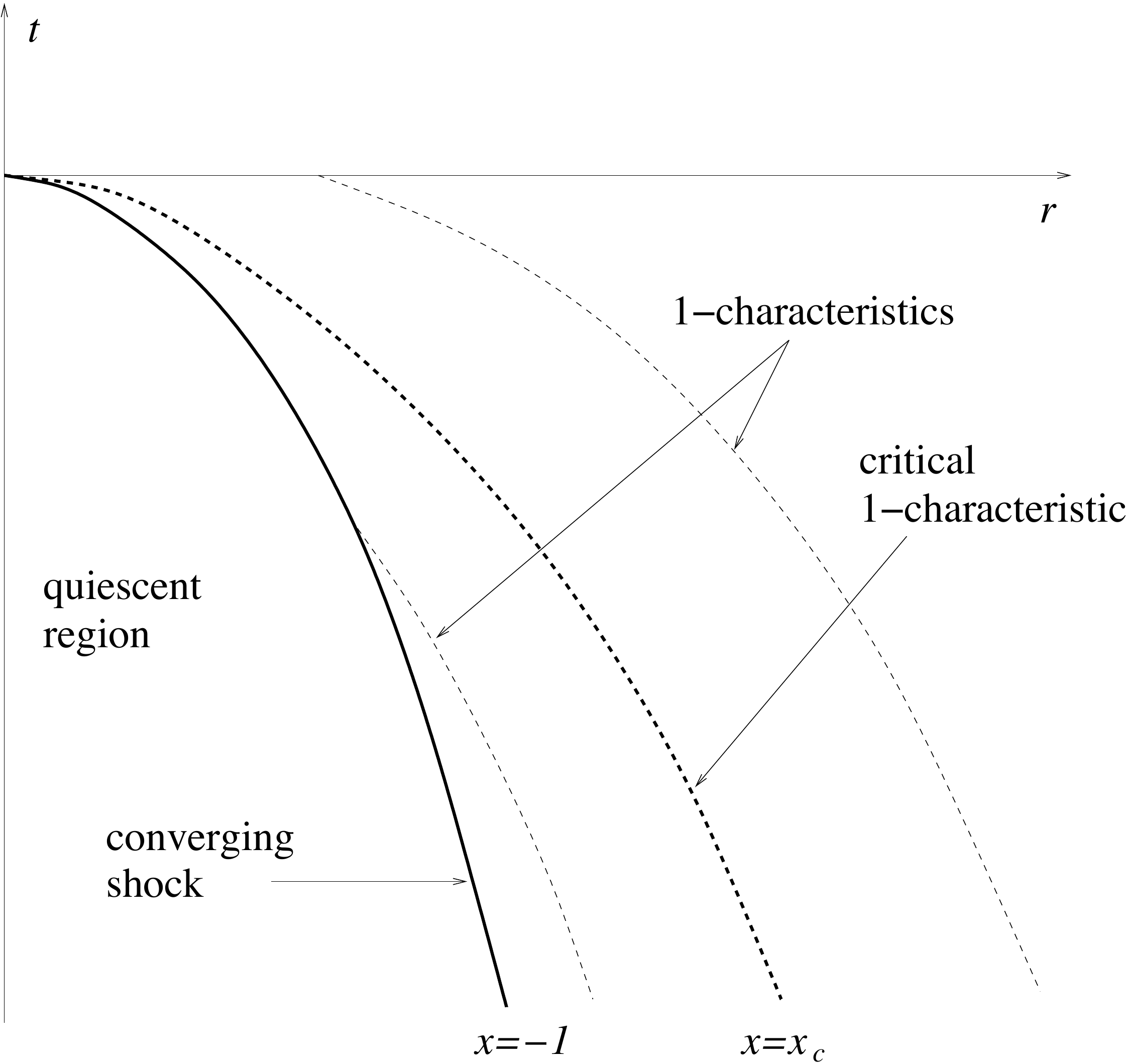}
	\caption{Converging similarity shock before collapse (schematic).}
\end{figure}

We point out that, in considering {\em weak} solutions, one should admit 
solutions with jumps in the derivatives of the flow variables across characteristics. 
In particular, $V$ and $C$ could enter and exit $x=x_c$ with different slopes. 
However, we shall not exploit this feature in the present work.

%%%%%%%%%%%%%%%%%%%%%%%%%%%%%%%%
\subsection{Similarity cavity solutions}\label{sim_cavs}
%%%%%%%%%%%%%%%%%%%%%%%%%%%%%%%%
For the case of a collapsing cavity we consider a spherical vacuum
region centered at the origin, surrounded by fluid moving radially inward. Assuming 
for now the existence of similarity flows \eq{V}-\eq{R} with this
structure, we assume that the vacuum-fluid interface 
follows the path $x=-1$ for negative times. Again we consider 
the case where this curve hits the origin 
with infinite speed at time $t=0$, so that $\lambda>1$. The interface 
is a particle trajectory, giving the initial condition for $V$ at $x=-1^+$ as
\beq\label{V_init_vac}
	V(-1)=-1.
\eeq
To select initial conditions for $R$ and $C$ at $x=-1$, 
we impose the further constraint that the entropy takes a fixed, 
constant value $\bar S$ throughout the fluid region for negative times (before a 
shock is reflected off the origin). The fluid pressure is then given 
by $A\rho^\gamma$, where $A=A(\bar S)$ is a constant. 
As the fluid pressure must vanish along the vacuum interface, it follows 
that the same holds for the density $\rho$, and also the sound speed 
$c=\sqrt{\gamma A\rho^{\gamma-1}}$. 
Equations \eq{C} and \eq{R} thus gives the initial conditions for
$C$ and $R$ at $x=-1^+$ as
\beq\label{CR}
	C(-1)=R(-1)=0.
\eeq
For later reference we note that isentropic similarity flow requires  
\beq\label{spec_kappa}
	\kappa=-\frac{2(\lambda-1)}{\gamma-1};
\eeq
this is a consequence of the momentum equation \eq{u} with 
$\rho c^2=\gamma A\rho^\gamma$, upon substituting for $u$ 
and $\rho$ from \eq{V} and \eq{R}, respectively.

It turns out that the similarity cavity flows constructed below
immediately leaves the starting point $(V,C)=(-1,0)$ by moving into the
region $C>1+V>0$. Just as for the shock case discussed above, 
we insist on ``well-behaved'' solutions satisfying \eq{V/x_C/x-zero}.
It follows that the cavity solution has to move back across the critical line
$\{C=1+V\}$, for some $x_c\in(-1,0)$, before continuing on toward 
the origin.

We note that, in contrast to the case of a similarity shock, 
in similarity cavity flow only the fluid velocity $u$ blows up along the 
curve $x\equiv -1$, while $c$, $\rho$, $p$, and $\theta$ all 
vanish there. On the other hand, \eq{V}-\eq{R} imply that all of
$u$, $c$, $\rho$, $p$, and $\theta$ blow up along all other 
curves $x\equiv constant \in (-1,0)$ as $t\uparrow 0$. (This last 
assertion requires that $V$ and $C$ does not vanish at any 
$x\in (-1,0)$; this will be the case for the similarity cavity flows 
constructed below.) Furthermore, the profiles for $u$, $c$, $p$,
and $\theta$ at time of collapse are again given by
\eq{uc_0}-\eq{ptheta_0} (provided the limits in \eq{V/x_C/x-zero} 
are non-zero, which holds for the cavity flows constructed below). 
Finally, for similarity cavity flow, also the density is unbounded 
at time $t=0$:
\[\rho(0,r)=R(0)r^\kappa,\] 
where $\kappa$, given by \eq{spec_kappa}, is strictly negative since 
$\lambda>1$. 

As the sound speed $c$ vanishes along the vacuum interface,
the characteristics degenerate there and become tangent to
the interface; a representative situation is recorded in Figure 2.
\begin{remark}
	It can be verified that the situation in Figure 2 is valid for 
	the cavity flows constructed below. In particular, \eq{dZdV}
	yields $C\sim \sqrt{1+V}$ near $x=-1$, and this implies that 
	any 1-characteristic between the interface $x=-1$ and the 
	critical characteristic $x=x_c$ will meet the interface at a 
	time strictly before collapse. It does so tangentially; at the 
	same point a 3-characteristic starts off tangentially into the 
	flow, as indicated. 
\end{remark}

%%%%%%%%%%%%%%%%%%%
%	FIGURE 
%%%%%%%%%%%%%%%%%%%
\begin{figure}\label{Figure_2}
	\centering
	\includegraphics[width=9cm,height=8cm]{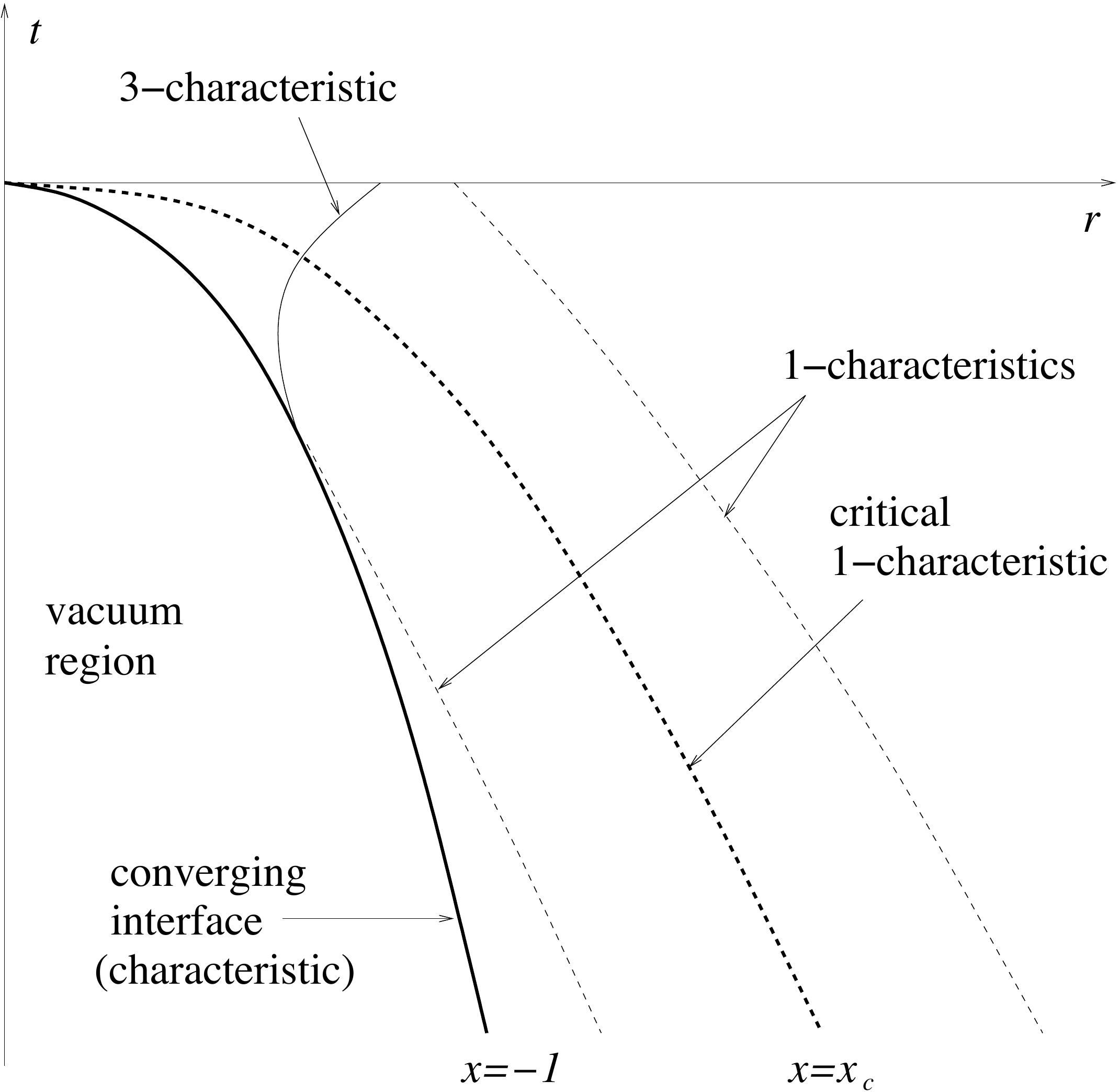}
	\caption{Similarity cavity flow before collapse (schematic).}
\end{figure}

%%%%%%%%%%%%%%%%%%%%%%%%%%%%%%%%%%%
\subsection{Similarity ODEs}\label{sim_ODEs}
%%%%%%%%%%%%%%%%%%%%%%%%%%%%%%%%%%%
Substituting \eq{V}-\eq{R} into \eq{u}-\eq{rho} we obtain a system of three
{\em similarity ODEs} for $V$, $C$, $R$. It is well-known that the constancy 
of specific entropy along particle trajectories provides one exact integral 
for the similarity ODEs (see \cite{rj}). Specifically, in any region where 
the flow is smooth, we have
\beq\label{exact_integral}
	R(x)^{q+1-\gamma}\left(\frac{C(x)}{x}\right)^2|1+V(x)|^q\equiv \text{constant},
\eeq
where 
\beq\label{power}
	q=\frac{\kappa(\gamma-1)+2(\lambda-1)}{\kappa+n},
\eeq
where $n=1,2,3$ is the spatial dimension. In the case of an incoming cavity,
the flow is isentropic for $t<0$, and $q$ vanishes according to \eq{spec_kappa}, 
while the right-hand side of \eq{exact_integral} is determined once the constant 
value $\bar S$ of the entropy is assigned. 

One can therefore obtain a closed
system for two of the unknowns, the standard choice being $V$ and $C$. 
The resulting ODEs are (see \cites{cf,L})
\begin{align}
	V'(x)&=-\frac{1}{\lambda x}\frac{G(V(x),C(x),\lam)}{D(V(x),C(x))}\label{V_sim2}\\
	C'(x)&=-\frac{1}{\lambda x}\frac{F(V(x),C(x),\lam)}{D(V(x),C(x))}\label{C_sim2}
\end{align}
where $'=\frac{d}{dx}$ and the polynomial functions $D$ and $G$, and the rational 
function $F$ are given by
\begin{align}
	D(V,C)&=(1+V)^2-C^2\label{D}\\
	G(V,C,\lam)&=C^2\left[nV+{\textstyle\frac{2(\lam-1)}{\gamma+s-1}}\right]-V(1+V)(\lam+V)\label{G}\\
	F(V,C,\lam)&=C\left\{C^2\left[1+{\textstyle\frac{s(\lam-1)}{\gamma(1+V)}}\right]
	-\left[1+{\textstyle\frac{(n-1)(\gamma-1)}{2}}\right](1+V)^2\right.\label{F}\\
	&\quad\qquad\left.+\left[{\textstyle\frac{(n-1)(\gamma-1)+(\gamma-3)(\lam-1)}{2}}\right](1+V)
	-{\textstyle\frac{(\gamma-1)(\lam-1)}{2}}\right\}. \nn
\end{align}
Here $s$ is a logical variable: $s=1$ for the shock case and $s=0$ for the cavity case.
Combining \eq{V_sim2} and \eq{C_sim2} we obtain a single ODE 
\beq\label{CV_ode}
	\frac{dC}{dV}=\frac{F(V,C,\lam)}{G(V,C,\lam)}
\eeq
relating $V$ and $C$ along similarity solutions.

%%%%%%%%%%%%%%%%%%%%%%%%%%%%%%%%%%%%%%%%%%%%%
%%%%%%%%%%%%%%%%%%%%%%%%%%%%%%%%%%%%%%%%%%%%%
\section{Construction of complete similarity flows}\label{constr_sim_solns}
%%%%%%%%%%%%%%%%%%%%%%%%%%%%%%%%%%%%%%%%%%%%%
%%%%%%%%%%%%%%%%%%%%%%%%%%%%%%%%%%%%%%%%%%%%%
In this section we discuss the existence of solutions to the similarity 
ODEs, and how these are used to build physically meaningful similarity 
shock and similarity cavity flows. We seek {\em complete} solutions defined for all times. 

The overall approach is, in principle, to 
solve \eq{CV_ode} for $C=C(V)$ with the appropriate initial data, and  
substitute the result into \eq{V_sim2}-\eq{C_sim2} to obtain 
$x$-parametrizations for $V=V(x)$ and $C=C(x)$ via quadrature. From these $R=R(x)$ 
can be determined from the exact integral in \eq{exact_integral}. For the discontinuous 
solutions under consideration, the Rankine-Hugoniot relations \eq{V_jump}-\eq{R_jump} 
are used. These will uniquely determine the value of the constant on the right-hand side of 
\eq{exact_integral} in each region where the solution is smooth. The 
original flow variables $\rho$, $u$, and $c$ are then given via \eq{sim_coord}-\eq{R}. 
Finally one needs to verify that the solution so obtained is 
physically acceptable. 

The analysis is complicated by the fact that the ODE \eq{CV_ode}
possesses a number of critical points (common zeros of $F$ and $G$), 
whose location varies with $\gamma$, $\lambda$, and $s$.
Furthermore, these may or may not be located on the critical 
lines 
\[\mathcal C_\pm:=\{C=\pm(1+V)\},\] 
along which the denominator $D$ in \eq{V_sim2} and \eq{C_sim2} vanishes. 
As discussed below, this is a key issue. 
Among the many treatments in the literature we find the work \cite{L} by Lazarus
to be the most useful for our needs. (Lazarus also studies solutions with several 
converging similarity shocks, a scenario we do not consider in the present work.)

The location of the initial data for $(V,C)$ at $x=-1$ implies that the
solutions of \eq{CV_ode} need to cross the critical line $\mathcal C_+$, 
before continuing on to the origin in the $(V,C)$-plane. 
%For the shock case this was argued for above in 
%Section \ref{sim_shocks}, with the crossing occurring for $x=x_c$, corresponding
%to crossing the critical characteristic in the $(r,t)$-plane. For the cavity case, the solution of \eq{CV_ode} 
%starts at $x=-1$ from the point critical point $(V,C)=(-1,0)$, which is a saddle point. 
%The cavity solutions will immediately move into the region $\{C>1+V>0\}$, and the requirement \eq{V/x_C/x-zero}
%shows that the solution $(V(x),C(x))$ must cross the critical line $\mathcal C_+$ before 
%reaching the origin at $x=0$.
Let 
\[\mathcal F:=\{(V,C)\,|\, F(V,C,\lambda)=0\},\qquad \mathcal F_\pm:=\mathcal F\cap\{C\gtrless 0\},\]
and define $\mathcal G$, $\mathcal G_\pm$ similarly by replacing $F(V,C,\lambda)$ by
$G(V,C,\lambda)$. As shown in \cite{L}, the set $\mathcal F\cap\mathcal G$ of critical 
points for \eq{CV_ode} can contain up to nine distinct points. One of these is $(V,C)=(-1,0)$, which
is the initial point for similarity cavity flow. In addition there may be up to two more critical points 
located on $\mathcal C_+$; we follow Lazarus' terminology and refer to these as points 6 and 8.
Now, a similarity flow must solve the full ODE system \eq{V_sim2} and \eq{C_sim2}. 
It follows from the form of these equations that any solution reaching the critical line $\mathcal C_+$, 
in order to continue on to the origin in the $(V,C)$-plane, must cross at a common zero of both 
$F$ and $G$. (Note that $F$ and $G$ are proportional along $\mathcal C_\pm$.)
It is this restriction that is used to determine what the relevant values of $\lambda$ can
be, for given values of $\gamma$, $n$, and $s$.

Lazarus \cite{L} provides a detailed analysis of the subtle issue of which 
$\lambda$-values give complete flows. In particular, Lazarus defines a function 
$\lambda_{std}=\lambda_{std}(\gamma,n,s)$ by the property that 
the solution of \eq{CV_ode}, with $\lambda=\lambda_{std}$ and starting at the 
appropriate initial point, passes {\em analytically} through point 6 or point 8.  
As pointed out in \cite{L}, most other 
authors have considered $\lambda_{std}$ to be the only physically relevant 
value of the similarity exponent.
Lazarus argues against this and shows that 
by removing the analyticity constraint one can, for fixed $\gamma$, 
$n$, and $s$, obtain whole families of complete similarity flows as $\lambda$ 
varies over certain non-trivial intervals. To obtain a complete breakdown of the 
possible cases requires numerical integration of the similarity ODEs. 
Most of the details of this analysis are included in \cite{L}. In particular,
the numerical values of $\lambda_{std}$ for $n=2,\, 3$ and $s=0,\, 1$ have 
been determined to several decimal places for a large number of 
$\gamma$-values (cf.\ Tables 6.2-6.5 in \cite{L}). According to Lazarus, 
``Numerically, it has been 
determined beyond question that it [i.e., the function $\lambda_{std}$] exists 
for the shock problem for all $\gamma>1$, and for the cavity problem for 
$\gamma>\gamma_{std}$.'' Here $\gamma_{std}$ depends on the spatial 
dimension and is approximately given by 2.9780 for $n=2$, and 2.4058
for $n=3$. In what follows we take these statements for granted.
Differently from many other cases of similarity solutions to PDEs, the similarity exponent $\lambda$ 
is not apriori given; no analytic expression for $\lambda_{std}$ is known.

Having determined those $\lambda$-values which gives relevant solutions 
to the similarity ODEs \eq{V_sim2} and \eq{C_sim2} for $x\in(-1,0)$, 
it remains to continue the solution through the origin and extend it to all 
$x>0$. As commented earlier, this is accomplished by inserting
an expanding similarity shock following a path of the form 
$r(t)=(\frac{t}{B})^\frac{1}{\lambda}$ for $t>0$ (i.e., $x\equiv B$, where 
$B>0$ is a constant).
The determination of $B$ and the construction of the solution for $x\in(B,\infty)$
are outlined in Section \ref{refl_shck} below; again, it appears necessary to 
do so through numerical integration of the equations. 

Having constructed a complete similarity shock or cavity solution in this manner,
it still remains to verify that the resulting flow is physically meaningful. 
This includes describing the solution behavior at the origin $r=0^+$ for 
$t>0$ (e.g., the velocity there should vanish), as well as checking that the 
mass, momentum, and total energy are locally bounded quantities. As we show in 
Section \ref{sim_weak_solns} (where we verify in detail that the similarity solutions
are genuine weak solutions to the Euler system),  the latter integral constraints require that the 
similarity exponent satisfies $\lambda<1+\frac{n}{2}$. It turns out that this is 
satisfied for all known values of $\lambda_{std}$ (cf.\ Tables 6.2-6.5 in \cite{L}).

While we agree with \cite{L} on the relevance of non-analytic similarity flows,
the more important point, for our purposes, is that we obtain {\em some} examples 
of shock and cavity flows that exhibit blowup. We therefore restrict attention to 
solutions corresponding to the ``analytic'' similarity exponent $\lambda_{std}$.

%%%%%%%%%%%%%%%%%%%%%%%%%%%%%%%%%%%%%%%%%%%%%
\subsection{Existence of similarity shock solutions prior to collapse}\label{existn_shock}
%%%%%%%%%%%%%%%%%%%%%%%%%%%%%%%%%%%%%%%%%%%%%
For the shock problem we first observe that, by construction, the converging shock
along $x=-1$ is compressive. The same holds for the diverging shock following
collapse. For the present 
case of an ideal gas, this implies that a fluid particle crossing the shock will suffer an increase 
in its physical entropy \cite{gr_96}; i.e., all discontinuities under consideration involving jumps of primary 
(undifferentiated) flow variables, are genuine, ``entropy-satisfying'' shocks.
Next, there is no issue near the 
initial point  $(V_1,C_1)$ given by the two first expressions in \eq{init_data}: 
the ODE \eq{CV_ode} is well behaved there and has a local solution 
for any values of $\lambda>1$ and $\gamma>1$. As outlined earlier, the solution must 
cross the critical line $\mathcal C_+=\{C=1+V\}$ before reaching $(V,C)=(0,0)$. 
As explained above we restrict attention to the particular value $\lambda=\lambda_{std}$
for which the solution crosses the critical line $\mathcal C_+$ in an analytic manner.
 
%%%%%%%%%%%%%%%%%%%%%%%%%%%%%%%%%%%%%%%%%%%%%
\begin{remark}\label{no_1_d_ex}
	The similarity ODEs \eq{V_sim2}-\eq{C_sim2} remain valid for $n=1$. 
	However, an analysis reveals that the solution starting out from 
	$(V_1,C_1)$ does not reach the critical line in this case, instead 
	ending at a critical point $(\bar V,\bar C)$ lying strictly above $\mathcal C_+$ 
	(this corresponds to  ``point 4'' in Lazarus' terminology \cite{L}).
	The same applies to the case of 1-d similarity cavity flow. 
	At $(\bar V,\bar C)$, $F(V,C)$ and $G(V,C)$ vanish and are Lipschitz continuous,
	while $D(V,C)$ does not vanish; therefore, the critical point is reached for $x=0$. 
	However, \eq{V} and \eq{C} then imply that the resulting flow is 
	physically meaningless at time of collapse in this case. 
	
	One could still attempt to build a 1-d flow exhibiting blowup by using only a 
	part of the similarity flow just described, say the part corresponding to 
	$x\in(-\infty,x_0)$, for an $x_0<0$. The idea would be to complete 
	the flow to all negative $x$, say, by a non-similarity 
	flow (e.g., a simple wave). However, any change made in the original similarity 
	flow for $x>x_0$ will necessarily influence the flow along the interface at $x=-1$, 
	{\em strictly} before $t=0$, and thus possibly prevent blowup. This is a 
	consequence of the fact that the original similarity solution does not reach the 
	critical line $\mathcal C_+$: there is no critical 1-characteristic in this case 
	(cf.\ Figure 2).
\end{remark}
%%%%%%%%%%%%%%%%%%%%%%%%%%%%%%%%%%%%%%%%%%%%%

After crossing the critical line the $\lambda_{std}$-solution approaches the origin 
$(V,C)=(0,0)$, which is a star point for \eq{CV_ode}. $F(V,C)$ and $G(V,C)$ both vanish and 
are Lipschitz continuous at the origin, while $D(V,C)$ does not vanish there.
It follows that the solution $(V(x),C(x))$ reaches the origin at $x=0$.
This critical point is again crossed in an analytic manner and the solution continues 
into the lower half of the $(V,C)$-plane; see Section \ref{refl_shck}.

\begin{remark}
	According to \eq{V/x_C/x-zero} the solution $(V(x),C(x))$ 
	approaches the origin with a slope $L/\ell$. For all cases we
	are aware of it is evident from numerical integration of the equations that 
	the limits in \eq{V/x_C/x-zero} are non-zero and finite. It follows from \eq{uc_0}
	that the flow in these cases is ``well-behaved'' and physically meaningful 
	at time of collapse.
\end{remark}

%%%%%%%%%%%%%%%%%%%%%%%%%%%%%%%%%%%%%%%%%%%%%
\subsection{Existence of cavity similarity solutions prior to collapse}\label{existn_cavity}
%%%%%%%%%%%%%%%%%%%%%%%%%%%%%%%%%%%%%%%%%%%%%
%%%%%%%%%%%%%%%%%%%%%%%%%%%%%%%%%%%%%%%%%%%%%
For the cavity  problem the initial point $(V,C)=(-1,0)$ for the ODE \eq{CV_ode} lies
on the critical line $\mathcal C_+=\{C=1+V\}$. This is a saddle point; a linearization about it in 
the variables $(V,Z=C^2)$ shows that there is a solution leaving along the 
direction 
\beq\label{dZdV}
	\frac{dZ}{dV}=\frac{\gamma(\gamma-1)(\lambda-1)}{n(\gamma-1)-2(\lambda-1)}.
\eeq
The solution $C(V)$ to \eq{CV_ode} therefore enters immediately the region 
$\{C>1+V>0\}$, provided $\lambda<1+\frac{n}{2}(\gamma-1)$, which we assume 
in what follows (for $s=0$).
\begin{remark}
	The corresponding solution $(V(x),C(x))$ of \eq{V_sim2}-\eq{C_sim2} 
	has $C(x)\to 0$ as $x\downarrow -1$. Note that 
	\eq{exact_integral} (with $q=0$) also gives $R(x)\to 0$ as $x\downarrow -1$.
	It follows that the density $\rho$ vanishes as the 
	interface $\{x=-1\}$ is approached from within the fluid. Therefore, 
	the constructed solution satisfies the physical boundary condition that 
	$p\propto \rho^{\gamma-1}$ vanishes along the vacuum interface.
\end{remark}
Further along the solution, the situation is similar to that for the shock case: the 
similarity exponent $\lambda$ must be chosen so that the solution of \eq{CV_ode} 
crosses the critical line $\mathcal C_+$ at a common zero of $F$ and $G$, i.e., through
one of the critical points labeled 6 or 8 in \cite{L}. 
Differently form the shock case, this will not occur for 
all values of $\gamma>1$. As noted earlier, for the cavity case, there is a minimal 
$\gamma_{std}(n)$ below which no value of $\lambda$ yields a solution with the 
required behavior.

After crossing the critical line $\mathcal C_+$, the situation is as in the 
shock case. The solution proceeds toward the origin in the 
$(V,C)$-plane, and passes through it in an analytical manner for $x=0$.

%%%%%%%%%%%%%%%%%%%%%%%%%%%%%%%%%%%%%%%%%%%%%
\subsection{Existence of similarity solutions beyond collapse; the reflected shock}\label{refl_shck}
%%%%%%%%%%%%%%%%%%%%%%%%%%%%%%%%%%%%%%%%%%%%%
The works \cites{L,RL_78,bg_96,RichtL_75,hun_60} consider the 
continuation of similarity shock and cavity solutions beyond collapse, 
to complete flows defined for all times. 
We are not aware of a general result addressing the unique continuation of solutions to
\eq{m_d_mass}-\eq{m_d_energy}, symmetric or not, for unbounded initial
data. On the other hand, it is reasonable to assume that no symmetry breaking 
occurs at time of collapse, and restrict attention to radial similarity flows with the 
same values of $\lambda$ and $\kappa$ also for $t>0$. 
Furthermore, the unbounded pressure distribution at time of collapse (cf.\ \eq{ptheta_0}) 
suggests searching for a solution in which an expanding shock wave is generated 
at the origin at time zero. 

Following \cites{L,RL_78}, we outline the construction of a reflected similarity 
shock propagating along a path $x=B=constant>0$. This 
shock will decay as it moves outward through the originally converging flow, leaving a 
non-isentropic flow region in its wake.
Providing a complete solution requires the continuation of the similarity solution 
$(V(x),C(x))$ of \eq{V_sim2}-\eq{C_sim2} found earlier beyond $x=0$, the determination 
of the reflected shock path (i.e., the value of $B$), and the solution of 
\eq{V_sim2}-\eq{C_sim2} for all $x>B$. The latter part of the solution provides the 
flow in the wake of the reflected shock; in particular, the asymptotic behaviors of $V(x)$ 
and $C(x)$ as $x\uparrow \infty$ yield the behavior of the flow variables at the center of motion 
($r=0$).

Continuing the solution $(V(x),C(x))$ through the star point (proper node) 
at the origin in the $(V,C)$-plane does not present any problem. This can be done in a 
unique analytic manner, and the solution $(V(x),C(x))$ is continued into the lower half-plane 
until it meets the critical line $\mathcal C_-=\{C=-1-V\}$.
Following \cite{L} we call this first part of the solution curve (in the lower half of 
the $(V,C)$-plane) ``arc (a).''

For each point $(\tilde V_0,\tilde C_0)$ on arc (a), we
then apply the Rankine-Hugoniot relations \eq{V_jump} and \eq{C_jump} to 
determine the unique point $(\tilde V_1,\tilde C_1)$, with $\tilde C_1<0$, 
to which the system can potentially jump.
(Recall that the form \eq{V_jump}-\eq{R_jump} of the Rankine-Hugoniot relations
assumes the discontinuity follows a ``similarity path'' $x=constant$, with the same 
values of $\lambda$, $\gamma$, and $\kappa$ on both sides of the discontinuity.)
As was noted in connection with \eq{V_jump}-\eq{R_jump}, since $\tilde C_0^2 < (1+\tilde V_0)^2$
along arc (a), the corresponding points $(\tilde V_1,\tilde C_1)$ necessarily 
lie below the critical line $\mathcal C_-$. 

As $x$ increases from $0$, the point $(\tilde V_0,\tilde C_0)\equiv (V(x),C(x))$ 
moves away from the origin along arc (a). At the same time the corresponding point 
$(\tilde V_1,\tilde C_1)$ traces out a certain simple curve; we follow \cite{L} and refer to 
it as the {\em jump locus} (of arc (a)).
(This jump locus is the smiley, dotted curve in the lower half plane indicated in Figure 3 below.)
According to \eq{V_jump}-\eq{C_jump} its left endpoint is $(V_1,-C_1)$ (corresponding to 
the point $(\tilde V_0,\tilde C_0)=(0,0)$), where 
$V_1$ and $C_1$ are given by \eq{init_data}. Its right end point lies on the critical 
line $\mathcal C_-$ and coincides with the end point of arc (a). 

At this stage, each point on the jump locus (except its endpoints) provides possible 
initial data for \eq{V_sim2}-\eq{C_sim2}, from which a solution trajectory 
should be continued for all $x>B$. 
The issue now is to argue that there is a unique point $(\hat V_1,\hat C_1)$ on the 
jump locus from which the solution can be continued to provide a
physically meaningful solution to \eq{m_d_mass}-\eq{m_d_energy}. 

A computation shows that the ODE \eq{CV_ode} has a critical point at $(V,C)=(V_0,-\infty)$, where
\beq\label{V_naught}
	V_0=-\frac{2(\lam-1)}{n(\gamma+s-1)}
\eeq
gives the vertical asymptote for the zero-level of $G(V,C,\lambda)$ in the $(V,C)$-plane.
This point corresponds to a saddle point at the origin in the variables $(v,\zeta)=(V-V_0,C^{-2})$.
There is therefore exactly one solution of \eq{CV_ode} which approaches the vertical 
asymptote $V=V_0$. Furthermore, it appears that this solution, when integrated in from infinity,
always lies entirely below the critical line $\mathcal C_-$, before intersecting the formerly determined
jump locus at a single point $(\hat V_1,\hat C_1)$. This solution trajectory is
referred to as ``arc (b).'' We then apply \eq{V_jump} and 
\eq{C_jump} to find the corresponding point $(\hat V_0,\hat C_0)$ on arc (a). 
The $x$-value $B$ at which the expanding shock is located is then determined by 
the condition that $(V(x),C(x))|_{x=B}=(\hat V_0,\hat C_0)$, where $(V(x),C(x))$ denotes the 
$x$-parametrization of arc (a). Modulo the $x$-parametrization 
of arc (b), this procedure determines the solution for all $x>0$, and provides a complete 
solution for both types of radial similarity flows. 

\begin{remark}\label{stagnation}
	As is evident from Figure 8.30 in \cite{L}, and explicitly pointed out in \cite{bg_96},
	for $\gamma\gtrsim3$ and $n=3$, the similarity shock solution suffers stagnation ($u=0$) 
	ahead of the reflected shock. In the phase plane this corresponds to the situation where 
	the solution $(V(x),C(x))$ moves along arc (a) into the left half plane 
	$\{V<0\}$ before jumping to arc (b).
\end{remark}

Before addressing the uniqueness of this solution, we
record how Lazarus \cite{L} obtains the $x$-parametrization of arc (b). 
First $V$ and $C$ are expanded 
in powers of the new independent variable $w=kx^{-\sigma}$, where $k$ and $\sigma>0$ 
are constants to be determined. With the ansatz 
\beq\label{V_Z_of_w}
	V(w)=\sum_{i=0}^\infty V_iw^i \qquad\text{and}\qquad C(w)=-\frac{1}{w}+\sum_{i=0}^\infty C_iw^i,
\eeq
substitution into \eq{V_sim2} and \eq{C_sim2} yields the value in \eq{V_naught} for $V_0$, and
\beq\label{sigma_z}
	\sigma=\frac{1}{\lambda}\Big[1+\frac{s(n-1)z}{1+V_0}\Big] \qquad 
	\text{where}\qquad z=\frac{\lambda-1}{(n-1)(\gamma+s-1)}.
\eeq
To integrate the ODE system in from the critical point $(V_0,-\infty)$ at infinity, Lazarus 
instead integrates the system for $V(w)$ and $C(w)$ from $w=0$, and thus obtains 
the $w$-parametrization of arc (b). This provides the value $w_1$ for which 
$(V(w),C(w))_{w=w_1}=(\hat V_1,\hat C_1)$, the point where arc (b) intersects
the jump locus of arc (a). As explained above, this determines, via the Rankine-Hugoniot relations 
\eq{V_jump}-\eq{C_jump} and the $x$-parametrization of arc (a), 
the location $x=B$ of the reflected shock. Finally, the $x$-parametrization of
arc (b) requires the determination of the constant $k$, which is now given by $k=B^\sigma w_1$.

%%%%%%%%%%%%%%%%%%%%%%%%%%%%%%%%%%%%
\begin{example}
In Figure 3 we have used Maple to display the complete similarity shock
solution ($s=1$) in the $(V,C)$-plane for the case $n=\gamma=3$. 
We have used the values $\lambda=\lambda_{std}(3,3,1)\approx 
1.5713126233$ and $B\approx 0.693970$ given by Table 6.5 in \cite{L}
(see erratum in \cite{L_errat}).  The solution starts at the starred point above the critical line 
$\{C=1+V\}$, moves downward, crosses $\{C=1+V\}$ and the origin smoothly, 
and then crosses the critical line $\{C=-1-V\}$ by jumping, before continuing along arc (b) 
toward the critical point at $(V_0,-\infty)$. Note that, in accordance with 
Remark \ref{stagnation}, the first jump point, corresponding to the state ahead of the 
reflected shock, is close to $\{V=0\}$.
%%%%%%%%%%%%%%%%%%
 
%%%%%%%%%%%%%%%%%%
\begin{figure}\label{Figure_3}
	\centering
	\includegraphics[width=9cm,height=9cm]{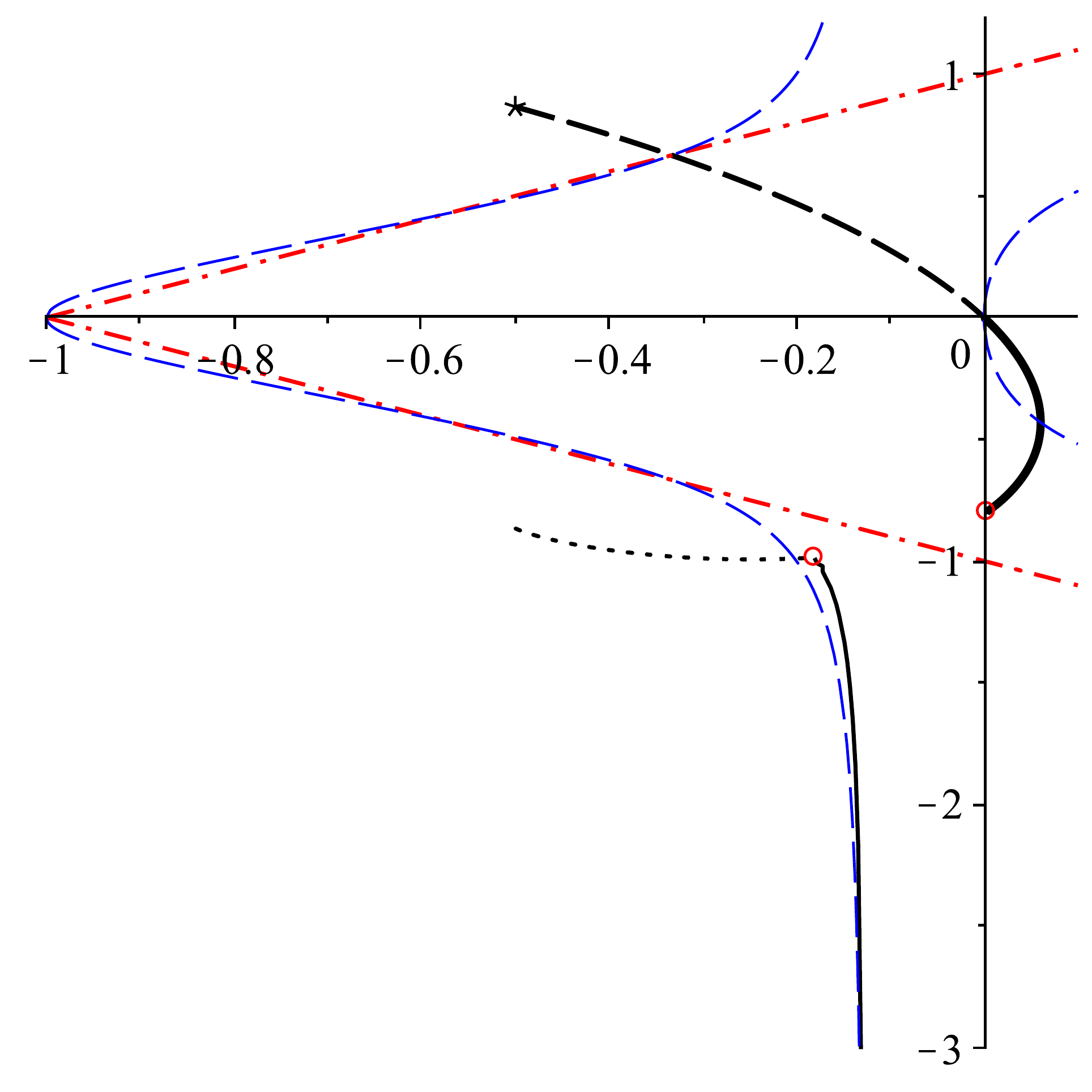}
	\caption{Complete trajectory of similarity shock solution ($n=\gamma=3$) in the $(V,C)$-plane. 
	Thick dash $=$ solution for $-1<x<0$, thick solid $=$ arc (a), dotted $=$ jump locus, solid $=$ arc (b),
	thin dash $=$ zero-level of $G(V,C)$, dash-dot $=$ critical lines, star $=$ starting point, circles $=$ 
	jump points.}
\end{figure}

\end{example}
%%%%%%%%%%%%%%%%%%%%%%%%%%%%%%%%%%%%

We note that, according to \eq{V}, the physical requirement that the particle velocity $u(t,r)$
vanishes at the center of motion $r=0$ for all $t>0$, imposes the condition 
$V(x)/x^\frac{1}{\lambda}\to 0$ as $x\uparrow\infty$. 
Of course, this is satisfied for the solution determined above since $V(x)$ in that case 
tends to the finite limit $V_0$ as $x\uparrow\infty$. 

By combining the asymptotic behavior of $V(x)$ and $C(x)$ with the exact integral 
\eq{exact_integral} we obtain that of $R(x)$, and thus a complete description of the 
flow near the center of motion. A calculation shows that the result depends on 
the value of $s$; at any fixed time $t>0$ and as $r\downarrow 0$, we have:
\begin{itemize}
	\item[(O1)] for similarity cavity flow ($s=0$): $\rho(t,r)$, $p(t,r)$, and 
	$\theta(t,r)\propto c(t,r)^2$ all tend to nonzero constants (cf.\ Figures 8.19-8.22 in \cite{L});
	\item[(O2)] for similarity shock flow ($s=1$): $\rho(t,r)\to 0$, $p(t,r)$ 
	tends to a strictly positive constant, while $c(t,r)$ and 
	$\theta(t,r)$ both tend to $+\infty$ (cf.\ Figures 8.25-8.28 in \cite{L}).
\end{itemize}
(For a representative calculation, see the proof of Lemma \ref{prelim} below.) 
It is noteworthy that, in the case of similarity shock flow, the density vanishes at the 
center of motion after collapse, without the pressure tending to zero there. For the 
ideal gas under consideration, this 
yields unbounded temperature and sound speed
at $r=0$ for $t>0$. (This contradicts Lazarus' statement on p.\ 330 in \cite{L}
when $s=1$.)
In our view, this is another manifestation of the borderline physicality 
of the radial similarity solutions under consideration.

It remains to discuss the uniqueness of the solution determined above, which
was obtained by exploiting the critical (saddle) point $(V_0,-\infty)$ at infinity for 
the ODE  \eq{CV_ode}. Consider first similarity cavity flow ($s=0$), in which case
\eq{CV_ode} has critical points also at $(-\infty,-\infty)$ and at $(\infty,-\infty)$.
However, neither of these appear to be reachable from the jump locus of arc (a).
Indeed, from the phase portraits it appears that all solution trajectories $(V(x),C(x))$ 
starting from points on the jump locus lying to the left of $(\hat V_1,\hat C_1)$ end up 
(for a finite value of $x$) on the critical line $\mathcal C_+$, while all trajectories
starting from points on the jump locus lying to the right of $(\hat V_1,\hat C_1)$ end up
on $\mathcal C_-$. There is no way to continue these solutions to all $x>0$ 
and obtain complete, physically meaningful flows.

For the case of similarity shock flow ($s=1$), the ODE \eq{CV_ode} has an 
additional critical point at $(V,C)=(-1,-\infty)$ (due to the $(1+V)^{-1}$-term 
in $F(V,C,\lambda)$ in this case, cf.\ \eq{F}). From the phase portraits it appears 
that all solution trajectories $(V(x),C(x))$ starting from points on the jump locus 
lying to the left of $(\hat V_1,\hat C_1)$ approaches this point. (All trajectories
starting from points on the jump locus lying to the right of $(\hat V_1,\hat C_1)$ 
appear again to end up on $\mathcal C_-$ for finite $x$-values). Changing to the variables 
$(V,\frac{1}{C})$ and linearizing, reveals that the point $(V,C)=(-1,-\infty)$ is necessarily 
reached for a finite $x$-value, say $\check x$ (depending on where along the jump 
locus the solution started). According to \cite{RL_78}, this shows that the critical point 
$(-1,-\infty)$ cannot describe the physical state at $r=0^+$ for $t>0$ (since this 
corresponds to $x=+\infty$), and is therefore irrelevant. 
However, this does not resolve the issue completely. A calculation shows that if
$(V(x),C(x))$ of \eq{V_sim2}-\eq{C_sim2} tends to $(-1,-\infty)$ as $x\uparrow \check x$, 
then the density $\rho(t,r)$ at a fixed time $t>0$ will satisfy
\[\rho(t,r)\downarrow 0\quad\text{as}\quad r\downarrow 
\big(\textstyle\frac{t}{\check x}\big)^\frac{1}{\lambda};\]
that is, a vacuum is reached.
%Namely, one could 
%construct a whole continuum of different similarity flows parametrized by the states along 
%arc (a): after passing through the origin in the $(V,C)$-plane, a solution would 
%move along  arc (a) to any given point corresponding to an $x$-value with $x<B$, 
%then jump to the jump locus of arc (a), and from there flow toward the vertical 
%line  $V=-1$, at which point a vacuum is reached for $x=\check x$. 
%Finally, the solution would be completed by continuing it as a vacuum to all $x>\check x$. 
This solution structure is not unreasonable: one might well 
imagine an expanding vacuum region opening up in the wake of a strong, 
expanding shock (a possibility considered by Hunter \cite{hun_60} for 
the particular case of similarity cavity flow with $\gamma=7$). However, a further 
calculation reveals that the pressure $p(t,r)$ does {\em not} tend to zero as 
$ r\downarrow (t/\check x)^{1/\lambda}$ (for $t>0$ fixed).
This type of solutions is therefore rejected as unphysical.
	
While these observations do not provide rigorous proof, they support the view that the 
only way to obtain a complete and physically admissible solution, 
is by having $(V(x),C(x))$ approach the 
saddle point at $(V_0,-\infty)$ as $x\uparrow \infty$. It therefore appears that both similarity shock 
and similarity cavity solutions are uniquely determined beyond collapse - at least among 
similarity flows.

%%%%%%%%%%%%%%%%%%%%%%%%%%%%%%%%%%%%%%%%%%%%%
\section{Weak and radial weak Euler solutions}\label{weak_solns} 
%%%%%%%%%%%%%%%%%%%%%%%%%%%%%%%%%%%%%%%%%%%%% 
We next consider whether the radial similarity solutions constructed
above, considered as function of time and space, provide weak solutions 
to the original multi-d Euler system \eq{m_d_mass}-\eq{m_d_energy}. 

For concreteness, in what follows, we focus on the case of similarity shock solutions, 
in which case the radial velocity, sound speed, pressure and temperature 
are unbounded at time of collapse, cf.\ \eq{uc_0}-\eq{ptheta_0}. 
The formulation and verification of the weak form of the equations 
therefore requires attention. Somewhat surprisingly this does not 
appear to have been addressed in the existing literature. 

In this section we formulate the weak form of the Euler system (in the 
absence of vacuum regions), first for general, multi-d solutions, and
then specialize it to the case of radial solutions.

%%%%%%%%%%%%%%%%%%%%%%%%%%%%%%%%%%%%%%%%%%%%%
\subsection{General, multi-d weak solutions}\label{multi-d_weak_solns} 
%%%%%%%%%%%%%%%%%%%%%%%%%%%%%%%%%%%%%%%%%%%%%
%Below we provide a definition of (non-vacuum) weak solutions to the compressible 
%Euler system. The definition we give requires some comments. First, for flows with 
%discontinuities there is the issue of admissibility. While not the only possible approach, 
%in this work we take the point of view that this is not an issue for the similarity flows 
%under consideration in this work. More precisely, in the case of a similarity shock 
%solution there is a single, converging discontinuity whose reflection off of the origin 
%yields a single, diverging discontinuity. These discontinuities are constructed to satisfy 
%the standard entropy condition (increase in entropy for fluid particles crossing the 
%discontinuity as time increases). The same applies to the reflected shock following 
%the collapse of a cavity. Next, in the weak form of the conservation laws, we shall 
%demand that the terms corresponding to the conserved quantities belong to 
%$C(\RR_t;L^1_{loc}(\RR^n)$, ensuring that the flow contains finite amounts
%of mass, momentum, and energy at all times. 
%We also require that the flux terms belong to $L^1_{loc}(\RR_t\times\RR^n)$. 
%Finally, for flows 
%with a vacuum region (as in similarity cavity solutions), we impose the weak form 
%of the equations only on the support of the density, and require that the pressure
%tend to zero as the fluid-vacuum interface is approached from within the fluid.

We write $\rho(t)$ for $\rho(t,\cdot)$ etc., $\vec u=(u_1,\dots,u_n)$, 
$u:=|\vec u|$, and let $z=(z_1,\dots,z_n)$ denote the spatial variable in $\RR^n$.
We restrict attention to non-vacuum solutions.
%%%%%%%%%%%%%%%%%%%%%%%%%%%%
\begin{definition}\label{weak_soln}
	Consider the compressible Euler system \eq{m_d_mass}-\eq{m_d_energy}
	in $n$ space dimensions, with a given pressure function $p=p(\rho,e)\geq 0$,  
	and let the measurable functions $\rho,\, u_1,\dots,u_n,\, e:\RR_t\times \RR_z^n\to \RR$ 
	be given. We say that these constitute a (non-vacuum) weak solution to 
	\eq{m_d_mass}-\eq{m_d_energy} provided that:
	\begin{itemize}
		\item[(i)] the functions $\rho$ and $e$ satisfy $\rho(t,z)>0$ 
		and $e(t,z)\geq 0$ for a.a.\ $(t,z)\in \RR\times \RR^n$;
		\item[(ii)] the maps $t\mapsto \rho(t)$, $t\mapsto \rho(t) u(t)$,
		and $t\mapsto \rho(t)(e(t)+\frac{u(t)^2}{2})$ belong to $C(\RR_t;L^1_{loc}(\RR^n_z))$;
		\item[(iii)]  the functions $\rho u^2$, $p$, and
		$\big[\rho \big(e+\textstyle\frac{u^2}{2}\big)+p\big]u$ belong to $L^1_{loc}(\RR_t\times\RR^n_z)$;
		\item[(iv)] the conservation laws for mass, momentum, and energy are 
		satisfied weakly in sense that
		\begin{align}
			\int_\RR\int_{\RR^n} \rho\vp_t+\rho\vec u\cdot\nabla_{z}\vp
			\, dzdt &=0\label{m_d_mass_weak}\\
			\int_\RR\int_{\RR^n} \rho u_i\vp_t
			 +\rho u_i\vec u\cdot\nabla_{z}\vp+p\vp_{z_i}\, dzdt &=0
			 \qquad \text{for $i=1,\dots, n$} \label{m_d_mom_weak}\\
			\int_\RR\int_{\RR^n} \rho \big(e+\textstyle\frac{u^2}{2}\big) \vp_t
			+\left[\rho \big(e+\textstyle\frac{u^2}{2}\big)+p\right] 
			\vec u\cdot\nabla_{z}\vp\, dzdt &=0 \label{m_d_energy_weak}
		\end{align}
		whenever $\vp\in C_c^1(\RR_t\times \RR^n_z)$ (the space of $C^1$-smooth functions with compact support).
	\end{itemize}
\end{definition}
\begin{remark}
Note that we allow for the possibility that the density vanishes on sets of 
measure zero. 
This is relevant since, as noted above, the similarity shock solutions constructed earlier include a 
vacuum state at the center of motion after collapse.

Also, we do not address admissibility of weak solutions. While not the 
only possible approach, we consider the similarity shock solutions under 
consideration to be admissible since their discontinuities are, by construction,
compressive shocks in ideal gases.
\end{remark}

%To define a weak solution to the initial value problem with data 
%$(\rho_0, \vec u_0, e_0)(z)$ prescribed at some initial time $t_0$, we 
%would restrict the time variable to $[t_0,\infty)$, require the maps in (ii)
%to belong to $C([t_0,\infty);L^1_{loc}(\RR^n_z))$ (with values $\rho_0$, $\rho_0u_0$,
%and $\rho_0(e_0+u^2_0/2)$ at time $t_0$, respectively), and add the appropriate 
%initial terms in \eq{m_d_mass_weak}-\eq{m_d_energy_weak}.
%
%We shall not pursue a general definition of weak solutions containing vacuum
%regions. This would require imposing the physical boundary condition 
%that the pressure vanishes as one approaches the vacuum from within the
%fluid. In turn, this raises issues about the regularity of the fluid-vacuum interface and the 
%sense in which the limit is taken. These problems are not present for the explicit 
%similarity cavity flows we consider below, for which the pressure vanishes in 
%a pointwise sense at each time along the interface.

%%%%%%%%%%%%%%%%%%%%%%%%%%%%

%%%%%%%%%%%%%%%%%%%%%%%%%%%%%%%%%%%%%%%%%%%%%
\subsection{Radial weak Euler solutions}\label{rad_weak_solns}
%%%%%%%%%%%%%%%%%%%%%%%%%%%%%%%%%%%%%%%%%%%%%
Next, for completeness we detail the relationship between weak solutions of the multi-d 
Euler system \eq{m_d_mass}-\eq{m_d_energy} and ``radial weak solutions'' of the radial 
version \eq{mass}-\eq{energy}. This analysis has been provided earlier by 
Hoff \cite{hoff} for radial solution of the compressible, isentropic Navier-Stokes 
system.  

%For functions the first argument is always the time variable, and the second argument 
%is always the spatial variable ($r$, $x$, or $z$); ditto for functions spaces; we 
Setting $m:=n-1$ we let 
\[\RR^+=(0,\infty),\qquad \RR_0^+=[0,\infty),\qquad 
L^1_{(loc)}(dt\times r^mdr)=L^1_{(loc)}(\RR\times\RR^+_0,dt\times r^mdr),\]
and $C^1_c(\RR\times\RR^+_0)$ denotes the set of real-valued functions 
$\psi(t,r)$ defined on $\RR\times\RR^+_0$ and with the property that $\psi$ is $C^1$ smooth 
on $\RR\times\RR^+_0$ and vanishes outside
$[-\bar t,\bar t]\times[0,\bar r]$ for some $\bar t,\, \bar r\in\RR^+$.
In particular, for any $\psi$ in $C^1_c(\RR\times\RR^+_0)$
the derivatives $\partial^l_t\partial_r^k \psi$ with $0\leq l+k \leq 1$ have well-defined (finite), 
continuous, and possibly non-vanishing, traces along the $t$-axis.
Finally, we let $C^1_0(\RR\times\RR^+_0)$ denote the set of functions 
$\psi\in C^1_c(\RR\times\RR^+_0)$ with the additional property that $\psi(t,0)\equiv 0$.

\begin{remark}\label{psi_0_rmk}
	It follows from this that for any $\psi\in C^1_0(\RR\times\RR^+_0)$, and any 
	compact time interval $[-T,T]$, there is a constant $A=A_{\psi,T}$ so that 
	\[|\psi(t,r)|\leq Ar\quad\text{for all $t\in[-T,T]$.}\]
\end{remark}

The relevance of these function classes is the following: when the weak 
formulation of the full multi-d Euler system 
\eq{m_d_mass}-\eq{m_d_energy} is applied to radial solutions, then the relevant 
``test functions'' for the radial continuity and energy equations 
will belong to $C^1_c(\RR\times\RR^+_0)$, while the relevant ``test functions'' for
the radial momentum equation will belong to $C^1_0(\RR\times\RR^+_0)$. 
Before verifying this we define ``radial weak solutions.''
%%%%%%%%%%%%%%%%%%%%%%%%%%%%
\begin{definition}\label{rad_symm_weak_soln}
	Consider the radial version \eq{mass}-\eq{energy} of the compressible Euler 
	system \eq{m_d_mass}-\eq{m_d_energy}, where $(t,r)$ ranges over $\RR\times \RR^+$ 
	and $p=p(\rho,e)\geq 0$ is a given pressure function. 
	
	Let the measurable functions $\rho,\, u,\, e:\RR_t\times \RR^+_r\to \RR$  
	be given. We say that these constitute a (non-vacuum) radial weak solution to 
	\eq{mass}-\eq{energy} provided that:
	\begin{itemize}
		\item[(i)] the functions $\rho$ and $e$ satisfy $\rho(t,r)>0$ 
		and $e(t,r)\geq 0$ for a.a.\ $(t,r)\in \RR\times \RR^+$;
		\item[(ii)] the maps $t\mapsto \rho(t)$, $t\mapsto \rho(t)u(t)$, and 
		$t\mapsto \rho(t)(e(t)+\frac{u(t)^2}{2})$ belong to
		$C(\RR_t;L^1_{loc}(r^mdr))$;
		\item[(iii)]  the functions $\rho u^2$, $p$, and
		$\big[\rho \big(e+\textstyle\frac{u^2}{2}\big)+p\big]u$ belong to $L^1_{loc}(dt\times r^mdr)$;
		\item[(iv)] the conservation laws for mass, momentum, and energy are 
		satisfied in the distributional sense that
		\begin{align}
			\int_{\RR}\int_{\RR^+} \left(\rho\psi_t+\rho u\psi_r\right)\, r^mdrdt &=0 
			\qquad\forall \psi\in C^1_c(\RR\times\RR^+_0) \label{radial_mass_weak}\\
			\int_{\RR}\int_{\RR^+} \left(\rho u\psi_t
			+\rho u^2\psi_r+p\big(\psi_r+\textstyle\frac{m\psi}{r}\big)\right)\, r^mdrdt &=0 
			\qquad\forall \psi\in C^1_0(\RR\times\RR^+_0)\label{radial_mom_weak}\\
			\int_\RR\int_{\RR^+} \left(\rho \big(e+\textstyle\frac{u^2}{2}\big) \psi_t
			+\left[\rho \big(e+\textstyle\frac{u^2}{2}\big)+p\right] u\psi_r\right)\, r^mdrdt &=0
			\qquad\forall \psi\in C^1_c(\RR\times\RR^+_0).\label{radial_energy_weak}
		\end{align}
	\end{itemize}
\end{definition}
%%%%%%%%%%%%%%%%%%%%%%%%%%%%
%%%%%%%%%%%%%%%%%%%%%%%%%%%%%%%%%
\begin{proposition}\label{rad_md}
	Consider the multi-d Euler system  \eq{m_d_mass}-\eq{m_d_energy} with a 
	given pressure function $p=p(\rho,e)$, together with its radially symmetric 
	version \eq{mass}-\eq{energy}. Then:
		given a radial weak solution 
		$(\tilde \rho, \tilde u, \tilde e)$ of \eq{mass}-\eq{energy},
		and setting 
		\beq\label{rad_to_gen_soln}
			\rho(t,z)=\tilde\rho(t,r)\qquad \vec u(t,z)
			=\tilde u(t,r)\frac{z}{r}\qquad e(t,z)=\tilde e(t,r)
			\qquad (r=|z|),
		\eeq 
		we obtain a  weak solution $(\rho,\vec u,e)$ of the multi-d Euler system 
		\eq{m_d_mass}-\eq{m_d_energy}.
%		\item[(b)] conversely, given a weak solution $(\rho,\vec u,e)$ of the multi-d 
%		Euler system \eq{m_d_mass}-\eq{m_d_energy}, and using 
%		\eq{rad_to_gen_soln} to define $(\tilde \rho, \tilde u, \tilde e)$,
%		we obtain a radial weak solution 
%		$(\tilde \rho, \tilde u, \tilde e)$ of \eq{mass}-\eq{energy}.
%	\end{enumerate}
\end{proposition}
%%%%%%%%%%%%%%%%%%%%%%%%%%%%%%%%%
\begin{proof}
	First, it is immediate that the properties in parts (i)-(iii) of Definition
	\ref{rad_symm_weak_soln}, together with \eq{rad_to_gen_soln}, 
	imply parts (i)-(iii) of Definition \ref{weak_soln}, respectively.
	It remains to verify the weak form of the equations.
%	
%	Also, it will suffice to consider part (a) of the theorem, as part (b) may be obtained by 
%	reversing the steps in the proof of the first part. [THIS LAST CLAIM IS NOT CLEAR...
%	SHOULD ADDRESS THE ISSUE OF HOW THE ``TEST FUNCTIONS'' IN THE RADIAL 
%	FORMULATIONS OF MASS AND MOMENTUM EQUATIONS YIELD STANDARD,
%	MULTI-D TEST FUNCTIONS BY ``SYMMETRIZATION'' (MASS CASE SEEMS 
%	STRAIGHTFORWARD, BUT WHAT ABOUT MOMENTUM EQN?)]
%	
%	
%	Thus, assume $(\tilde \rho, \tilde u, \tilde e)$ is a radial weak solution
%	of \eq{mass}-\eq{energy} and define $(\rho,\vec u,e)$ according to \eq{rad_to_gen_soln}.
	To verify \eq{m_d_mass_weak} we fix $\vp\in C_c^1(\RR\times \RR^n)$ and set
	\beq\label{psi_1}
		\psi(t,r):=\int_{|y|=1}\vp(t,ry)\, dS_{y}.
	\eeq
	Then $\psi\in C^1_c(\RR\times\RR^+_0)$ and \eq{radial_mass_weak} gives
	\begin{align*}
		0&=\int_{\RR}\int_{\RR^+} \left(\tilde \rho\psi_t+\tilde \rho \tilde u\psi_r\right)\, r^mdrdt\\
		&=\int_{\RR}\int_{\RR^+} \Big[\tilde \rho \int_{|y|=1}\vp_t(t,ry)\, dS_{y}
		+\tilde \rho \tilde u\int_{|y|=1}\partial_r\left(\vp(t,ry)\right)\, dS_{y}\Big]\, 
		r^mdrdt\\
		&=\int_{\RR}\int_{\RR^+}\int_{|y|=1} \Big[\tilde \rho\vp_t(t,ry)
		+\tilde \rho \tilde u \nabla_{z}\vp(t,ry)\cdot y\Big]\, r^m dS_{y}drdt
		=\int_\RR\int_{\RR^n} \rho\vp_t+\rho\vec u\cdot\nabla_{z}\vp \, dzdt,
	\end{align*}
	verifying the weak form \eq{m_d_mass_weak} of the continuity equation 
	\eq{m_d_mass} in the multi-d Euler system.
	
	Next, to verify \eq{m_d_mom_weak} we fix $i$ ($1\leq i\leq n$) and 
	$\vp\in C_c^1(\RR\times \RR^n)$, and set
	\beq\label{psi_2}
		\psi(t,r):=\int_{|y|=1}y_i\vp(t,ry)\, dS_{y}.
	\eeq
	Then $\psi\in C^1_0(\RR\times\RR^+_0)$ and \eq{radial_mom_weak} gives
	\beq\label{m2}
		\int_{\RR}\int_{\RR^+} \Big(\underbrace{\tilde \rho \tilde u\psi_t}_{I}
		+\underbrace{\tilde \rho \tilde u^2\psi_r}_{I\!I}
		+\underbrace{\tilde p\big(\psi_r+{\textstyle\frac{m\psi}{r}}\big)}_{I\!I\!I}\Big)\, r^mdrdt=0,
	\eeq
	where $\tilde p=p(\tilde \rho,\tilde e)$. Treating each term in turn, we have:
	\begin{align*}
		I&=\int_{\RR}\int_{\RR^+} \tilde \rho \tilde u\psi_t\, r^mdrdt
		=\int_{\RR}\int_{\RR^+} \tilde \rho \tilde u\Big[ \int_{|y|=1}
		y_i\vp(t,ry)\, dS_{y}\Big]_t\, r^mdrdt\\
		&=\int_{\RR}\int_{\RR^+}\int_{|y|=1} \tilde \rho \tilde u y_i \vp_t(t,ry)\, 
		r^m dS_{y}drdt
		=\int_\RR\int_{\RR^n} \rho u_i\vp_t \, dzdt,
	\end{align*}
	and
	\begin{align*}
		I\!I&=\int_{\RR}\int_{\RR^+} \tilde \rho \tilde u^2\psi_r\, r^mdrdt
		=\int_{\RR}\int_{\RR^+} \tilde \rho \tilde u^2\Big[ \int_{|y|=1}
		y_i\vp(t,ry)\, dS_{y}\Big]_r\, r^mdrdt\\
		&=\int_{\RR}\int_{\RR^+}\int_{|y|=1} 
		\tilde \rho \tilde u^2 y_i \nabla_{z}\vp(t,ry)
		\cdot y\,  r^m dS_{y}drdt
		=\int_\RR\int_{\RR^n} \rho u_i\vec u\cdot\nabla_{z}\vp \, dzdt.
	\end{align*}
	For $I\!I\!I$ we first calculate 
	\begin{align*}
		\left(r^m\psi\right)_r&=\partial_r\Big(r^m\int_{|y|=1} y_i\vp(t,ry)\, 
		dS_{y}\Big)
		= \partial_r\Big(\int_{|z|=r} \vp(t,z){\textstyle\frac{z_i}{|z|}}\, dS_{z}\Big)\\
		&= \partial_r\Big(\int_{|z|\leq r} \vp_{z_i}(t,z)\, 
		dz\Big)
		= \partial_r\Big(\int_0^r\int_{|y|=1} \vp_{z_i}(t,sy)\, s^m dS_{y}ds\Big)
		= r^m\int_{|y|=1}\vp_{z_i}(t,ry)\, dS_{y}.
	\end{align*}
	Using this we obtain that
	\begin{align*}
		I\!I\!I &=\int_{\RR}\int_{\RR^+} \tilde p\big(\psi_r+{\textstyle\frac{m\psi}{r}}\big)\, r^mdrdt
		=\int_{\RR}\int_{\RR^+} \tilde p\left(r^m\psi\right)_r\, drdt\\
		&=\int_{\RR}\int_{\RR^+}\int_{|y|=1} \tilde p \vp_{z_i}(t,ry) 
		r^m\, dS_{y}drdt
		=\int_\RR\int_{\RR^n} p\vp_{z_i} \, dzdt.
	\end{align*}
	Substituting these expressions for $I$, $I\!I$, and $I\!I\!I$ back into \eq{m2}, shows that the 
	weak form \eq{m_d_mom_weak} of the momentum equation \eq{m_d_mom} in the multi-d 
	Euler system is satisfied.
	
	Finally, to verify \eq{m_d_energy_weak} we fix $\vp\in C_c^1(\RR\times \RR^n)$ and again
	define $\psi(t,r)$ by \eq{psi_1}. Then $\psi\in C^1_c(\RR\times\RR^+_0)$ and 
	\eq{radial_energy_weak} gives
	\begin{align*}
		0&=\int_\RR\int_{\RR^+} \Big(\tilde \rho \big(\tilde e+{\textstyle\frac{\tilde u^2}{2}}\big) \psi_t
		+\Big[\tilde \rho \big(\tilde e+{\textstyle\frac{\tilde u^2}{2}}\big)+\tilde p\Big] \tilde u\psi_r\Big)\, 
		r^mdrdt\\
		&=\int_{\RR}\int_{\RR^+} \Big\{\tilde \rho \big(\tilde e+{\textstyle\frac{\tilde u^2}{2}}\big) 
		\int_{|y|=1}\vp_t(t,ry)\, dS_{y}\\
		&\qquad\qquad\qquad\qquad
		+\Big[\tilde \rho \big(\tilde e+{\textstyle\frac{\tilde u^2}{2}}\big)+\tilde p\Big] 
		\tilde u\int_{|y|=1}
		\partial_r\left(\vp(t,ry)\right)\, dS_{y}\Big\}\, r^mdrdt\\
		&=\int_{\RR}\int_{\RR^+} \int_{|y|=1}\Big\{\tilde \rho \big(\tilde e
		+{\textstyle\frac{\tilde u^2}{2}}\big) \vp_t(t,ry)\\
		&\qquad\qquad\qquad\qquad
		+\Big[\tilde \rho \big(\tilde e+{\textstyle\frac{\tilde u^2}{2}}\big)+p\Big] \tilde u
		y\cdot \nabla_{z}\vp(t,ry)\Big\}\, r^mdrdt\\
		&=\int_{\RR}\int_{\RR^n} \rho \big(e+{\textstyle\frac{u^2}{2}}\big) \vp_t
		+\Big[\rho \big(e+{\textstyle\frac{u^2}{2}}\big)+p\Big]\vec u\cdot \nabla_{z}\vp\, 
		dzdt,
	\end{align*}
	verifying the weak form \eq{m_d_energy_weak} of the energy equation 
	\eq{m_d_energy} in the multi-d Euler system.
\end{proof}
%%%%%%%%%%%%%%%%%%%%%%%%%%%%%%
\begin{remark}
	Note that the ``test function'' $\psi$ in \eq{psi_1} typically has non-vanishing trace 
	along the $t$-axis (e.g., when $n=3$, $\psi(t,r)\to 4\pi \cdot \vp(t,0)$ as $r\downarrow 0$), 
	while its $r$-gradient does vanish as $r\downarrow 0$.
	Also, the ``test-function'' $\psi$ in \eq{psi_2} behaves in 
	the opposite manner: $\psi(t,r)\to 0$ as $r\downarrow 0$, while typically 
	$\psi_r(t,r)\not\to 0$ as $r\downarrow 0$.
\end{remark}
%%%%%%%%%%%%%%%%%%%%%%%%%%%%%%	

%%%%%%%%%%%%%%%%%%%%%%%%%%%%%%%%%%%%%%
%%%%%%%%%%%%%%%%%%%%%%%%%%%%%%%%%%%%%%
\section{Similarity shock solutions as radial weak solutions}\label{sim_weak_solns} 
%%%%%%%%%%%%%%%%%%%%%%%%%%%%%%%%%%%%%%
%%%%%%%%%%%%%%%%%%%%%%%%%%%%%%%%%%%%%%
In this section we return to the case of an ideal gas and consider 
the similarity shock solutions constructed in Section \ref{constr_sim_solns} 
as candidates for weak solutions of the Euler system. The main result 
is that these provide {\em bona fide} weak solution that suffer blowup 
of primary flow variables at collapse. This conclusion holds for flows in two 
and three space dimensions provided the similarity shock solution $(R(x),V(x),C(x))$ satisfies 
the properties listed in (P1)-(P3) below. We stress that numerical computations 
clearly indicate that these properties are satisfied for the ``standard'' similarity 
solutions with $	\lambda=\lambda_{std}(\gamma,n,1)$, for a large range of 
$\gamma$-values (see Tables 6.4-6.5 in \cite{L}).
\begin{itemize}
	\item[(P1)] the function $1+V(x)$ is uniformly bounded below away from zero, 
	and from above, as $x$ varies over all of $\RR$;
	\item[(P2)] the limits $\ell$ and $L$ in \eq{V/x_C/x-zero} satisfy $-\infty<L<0<\ell<\infty$;
	\item[(P3)] $(V(x),C(x))\to(V_0,-\infty)$ as $x\uparrow \infty$, where $V_0$ is given by \eq{V_naught}.
\end{itemize}
We now fix $n=2$ or $n=3$ and let $s=1$, such that $\kappa$ in \eq{R} vanishes,
and $\rho(t,r)=R(x)$. 
%%%%%%%%%%%%%%%%%%%%%%%%%%%%%%
\begin{lemma}\label{prelim}
	With $n=2$ or $3$, and $1<\lambda<1+\frac{n}{2}$, assume (P1)-(P3) are satisfied for the 
	solution $(R(x),V(x),C(x))$ under consideration.
	Then $R(x)>0$ for all $x\in\RR$, the functions $R(x)$, $V(x)$, $V(x)/x$ are globally 
	bounded on $\RR$, and the functions $R(x)$, $V(x)/x$, $C(x)/x$ are continuous 
	at $x=0$. Finally, the function $R(x)(C(x)/x)^2$ is globally bounded. 
\end{lemma}
%%%%%%%%%%%%%%%%%%%%%%%%%%%%%%
\begin{proof}
	Clearly, (P1) and (P2) imply global boundedness of $V(x)$, continuity of
	$V(x)/x$, $C(x)/x$ at $x=0$ (when the latter two functions are defined to 
	take values $\ell$ and $L$ there, respectively), and therefore also global 
	boundedness of $V(x)/x$.
	
	Next, linearization of the ODE \eq{CV_ode} about $(V_0,-\infty)$ 
	shows that the leading order behaviors of $V$ and $C$ there 
	are given by \eq{V_Z_of_w}-\eq{sigma_z}:
	\beq\label{asymp_vals}
		V(x)\sim V_0= -\frac{2(\lam-1)}{\gamma n}\qquad\text{and}\qquad 
		C(x)\sim -x^\sigma \qquad\text{as $x\uparrow\infty$},
	\eeq
	where 
	\beq\label{sigma_q_shock}
		\sigma=\frac{1}{\lambda}\Big(1+\frac{\lambda-1}{\gamma-q}\Big)\qquad\text{with}\qquad
		q=\frac{2(\lambda-1)}{n}.
	\eeq
	We note that the constraint $\lambda<1+\frac{n}{2}$ implies $-1<V_0<0$, and thus
	\beq\label{reln}
		\gamma-q\equiv\gamma(1+V_0)>0.
	\eeq
	Also recall that the function $R(x)$ takes the value $1$ for $x<-1$; a calculation 
	using the Rankine-Hugoniot relations \eq{V_jump}-\eq{R_jump} together 
	with \eq{exact_integral}, shows that $R(x)>0$ for all $x>-1$ as well. By \eq{exact_integral},
	the continuity of $V(x)$ and $C(x)/x$ at $x=0$ implies that of $R(x)$. 
	According to \eq{exact_integral} we also obtain
	\beq\label{R_bnd}
		R(x)\sim\big(\textstyle\frac{C(x)}{x}\big)^{-\frac{2}{q+1-\gamma}}\sim x^{-\frac{2}{\gamma-q}(1-\frac{1}{\lambda})}
		\qquad\text{as $x\uparrow\infty$.}
	\eeq
	Thus, according to \eq{reln}, we have that $R(x)$ tends to zero as $x\uparrow\infty$ (establishing the 
	first part of (O2) in Section \ref{refl_shck}); it is therefore globally bounded. Finally, a similar 
	calculation shows that 
	\beq\label{aux}
		R(x)\Big|\frac{C(x)}{x}\Big|^2\sim x^{-2(1-\frac{1}{\lambda})}.
	\eeq
	Together with the continuity of $C(x)/x$ at $x=0$, this shows that $R(x)(C(x)/x)^2$ is globally bounded.
\end{proof}

For the solution $(R(x),V(x),C(x))$ under consideration we now define $\rho$, $u$, $c$, and $e$
via \eq{V}-\eq{R} and \eq{sound}. 
%%%%%%%%%%%%%%%%%%%%%%%%%%%%%%%%%%%%%%%%%%%%%
\begin{theorem}\label{main_thm}
        With $n=2$ or $3$, and under the assumption that (P1)-(P3) hold, the triple $(\rho,u,e)$ constitutes a radial 
        weak solution to \eq{mass}-\eq{energy}, with ideal pressure 
        law \eq{perf}, according to Definition \ref{rad_symm_weak_soln} whenever
        \beq\label{max_lam}
        	1<\lambda<1+\textstyle\frac{n}{2}.
        \eeq
\end{theorem}
%%%%%%%%%%%%%%%%%%%%%%%%%%%%%%%%%%%%%%%%%%%%%
%\begin{remark}
%	The constraint \eq{max_lam} is satisfied for all calculated values of 
%	$\lambda_{std}(\gamma,n,1)$ for $n=2$, $n=3$, and for $\gamma$
%	between $1$ and $9999$; .
%\end{remark}
According to Proposition \ref{rad_md}, it follows that these solutions
provide (non-vacuum) weak solutions of the multi-d Euler system 
\eq{m_d_mass}-\eq{m_d_energy}, according to Definition \ref{weak_soln}, 
with unbounded amplitudes.

The proof of Theorem \ref{main_thm} is organized as follows. First,
part (i) of Definition \ref{rad_symm_weak_soln} is 
immediate from Lemma \ref{prelim} and the definitions of $\rho$ and $e$.
The next two subsections consider 
the continuity and integrability requirements in parts (ii) and (iii) of 
Definition \ref{rad_symm_weak_soln}, respectively. Subsection \ref{weak_forms} finishes the proof by
analyzing the weak form of the equations (part (iv) of Definition \ref{rad_symm_weak_soln}).

%%%%%%%%%%%%%%%%%%%%%%%%%%%%%%%%%%%
\subsection{Continuity and local integrability}\label{cont_and_loc_integr}
%%%%%%%%%%%%%%%%%%%%%%%%%%%%%%%%%%%
For a fixed $\br>0$ and with
\[M(t;\bar r):=\int_0^{\br} \rho(t,r)r^m\, dr,\qquad I(t;\bar r):=\int_0^{\br} \rho(t,r)|u(t,r)|r^m\, dr,\]
\[E(t;\bar r):=\int_0^{\br} \rho(t,r)e(t,r)r^m\, dr+\frac{1}{2}\int_0^{\br}\rho(t,r)u^2(t,r)r^m\, dr
=:E_P(t;\bar r)+E_K(t;\bar r),\]
the issue is to show that the maps $t\mapsto M(t;\bar r)$, $t\mapsto I(t;\bar r)$,
and $t\mapsto E(t;\bar r)$ are continuous at all times $t\in\RR$. 
Recall that the incoming and outgoing shock waves follow the paths 
$r=r_i(t)= (-t)^{1/\lambda}$ and $r=r_o(t)=(t/B)^{1/\lambda}$, 
respectively. In what follows we consider times $t$ small enough that $r_i(t)<\br$
if $t<0$ and $r_o(t)<\br$ if $t>0$. The calculations for the other cases are 
simpler and do not change the conclusions. We set
\[\alpha:=\frac{n}{\lambda}.\]

%%%%%%%%%%%%%%%%%%%%%%%%%%%%%%%%%%%
\subsubsection{Continuity of $M(t;\bar r)$}
%%%%%%%%%%%%%%%%%%%%%%%%%%%%%%%%%%%
For $t<0$ we have $\rho(t,r)=1$ for $0<r<r_i(t)$, such that
\beq\label{M_neg_t}
	M(t;\bar r)=\int_{0}^{r_i(t)} r^m\, dr+\int_{r_i(t)}^{\br} \rho(t,r)r^m\, dr
	=\frac{|t|^{\alpha}}{n}+\frac{1}{\lambda}|t|^{\alpha}
	\int_{-1}^{\frac{t}{\br^\lambda}}R(x)\, 
	\frac{dx}{|x|^{\alpha+1}},
\eeq
while for $t>0$ we have
\beq\label{M_pos_t}
	M(t;\bar r)=\Big[\int_0^{r_o(t)} + \int_{r_o(t)}^{\br} \Big] \rho(t,r)r^m\, dr
	= \frac{1}{\lambda}t^{\alpha}
	\Big[\int_{\frac{t}{\br^\lambda}}^B+\int_B^{\infty}\Big]R(x)\, 
	\frac{dx}{x^{\alpha+1}}.
\eeq
As $R(x)$ is globally bounded, the integrals in \eq{M_neg_t} and 
\eq{M_pos_t} are all finite, and  
$t\mapsto M(t;\bar r)$ is continuous at all times $t\neq0$.
For  $t=0$ we have
\beq\label{M_at_t_0}
	M(0;\bar r)=\frac{\br^{n}}{n}R(0).
\eeq
Observe that, as $R(x)$ is globally bounded, the second integral 
on the right-hand side of \eq{M_pos_t} and the first term on  the right-hand side of \eq{M_neg_t}
are of order $|t|^{\alpha}$, and thus vanish when $t\downarrow 0$ and $t\uparrow 0,$ respectively. 
Therefore, continuity from above at $t=0$ 
of $M(t;\bar r)$ follows once it is established that 
\[\frac{1}{\lambda}t^{\alpha}
	\int_{\frac{t}{\br^\lambda}}^B R(x)\, 
	\frac{dx}{x^{\alpha+1}}
	\to M(0;\bar r)\qquad\text{as $t\downarrow 0$.}\]
This may be verified by using L'H\^opital's rule and the continuity of the 
map $x\mapsto R(x)$ at $x=0$. The same argument shows that 
\[\frac{1}{\lambda}|t|^{\alpha}
	\int_{-1}^{\frac{t}{\br^\lambda}}R(x)\, 
	\frac{dx}{|x|^{\alpha+1}}
	\to M(0;\bar r)
	\qquad\text{as $t\uparrow 0$}\]
as well. Thus, the map $t\mapsto M(t;\bar r)$ is continuous 
at all times.

%%%%%%%%%%%%%%%%%%%%%%%%%%%%%%%%%%%
\subsubsection{Continuity of $I(t;\bar r)$}
%%%%%%%%%%%%%%%%%%%%%%%%%%%%%%%%%%%
For $t<0$ we have $u(t,r)=0$ for $0<r<r_i(t)$ such that
\beq\label{I_neg_t}
	I(t;\bar r)=\int_{r_i(t)}^{\br} \rho(t,r)|u(t,r)|r^m\, dr
	=\frac{1}{\lambda^2}|t|^{\alpha-1+\frac{1}{\lambda}}
	\int_{-1}^{\frac{t}{\br^\lambda}}R(x)\frac{|V(x)|}{|x|}\, 
	\frac{dx}{|x|^{\alpha+\frac{1}{\lambda}}},
\eeq
while for $t>0$ we have
\beq\label{I_pos_t}
	I(t;\bar r)=\Big[\int_0^{r_o(t)} + \int_{r_o(t)}^{\br} \Big] \rho(t,r)|u(t,r)|r^m\, dr
	= \frac{1}{\lambda^2}t^{\alpha-1+\frac{1}{\lambda}}
	\Big[\int_{\frac{t}{\br^\lambda}}^B+\int_B^{\infty}\Big]R(x)\frac{|V(x)|}{x}\, 
	\frac{dx}{x^{\alpha+\frac{1}{\lambda}}}.
\eeq
As $R(x)$ and $V(x)/x$ are globally bounded, and $\alpha+\frac{1}{\lambda}>1$ 
(by assumption \eq{max_lam}), the integrals in \eq{I_neg_t} and 
\eq{I_pos_t} are all finite, and $t\mapsto I(t;\bar r)$ is continuous 
at any time $t\neq0$. For $t=0$ we have, by property (P2) and 
with $\ell$ given by \eq{V/x_C/x-zero},
\beq\label{I_at_t_0}
	I(0;\bar r)=\frac{1}{\lambda}R(0)\ell\frac{\br^{n+1-\lambda}}{n+1-\lambda}.
\eeq
As the second term on  the right-hand side of \eq{I_pos_t} is of order 
$t^{\alpha-1+\frac{1}{\lambda}}$, and thus vanishes when $t\downarrow 0$ 
(by \eq{max_lam}), the continuity of $I(t;\bar r)$ from above at $t=0$ 
follows once it is established that 
\[\frac{1}{\lambda^2}t^{\alpha-1+\frac{1}{\lambda}}
	\int_{\frac{t}{\br^\lambda}}^B R(x)\frac{|V(x)|}{x}\, 
	\frac{dx}{x^{\alpha+\frac{1}{\lambda}}}
	\to I(0;\bar r)\qquad\text{as $t\downarrow 0$.}\]
This may be verified by using L'H\^opital's rule and the continuity of the 
map $x\mapsto R(x)\frac{|V(x)|}{x}$ at $x=0$. The same argument shows that 
\[\frac{1}{\lambda^2}|t|^{\alpha-1+\frac{1}{\lambda}}
	\int_{-1}^{\frac{t}{\br^\lambda}}R(x)\frac{|V(x)|}{|x|}\, 
	\frac{dx}{|x|^{\alpha+\frac{1}{\lambda}}}
	\to I(0;\bar r)
	\qquad\text{as $t\uparrow 0$}\]
as well. Thus, the map $t\mapsto I(t;\bar r)$ is continuous 
at all times.

%%%%%%%%%%%%%%%%%%%%%%%%%%%%%%%%%%%
\subsubsection{Continuity of $E(t;\bar r)$}
%%%%%%%%%%%%%%%%%%%%%%%%%%%%%%%%%%%
We consider first the kinetic energy
\[E_K(t;\br)=\frac{1}{2}\int_0^{\br} \rho(t,r)u^2(t,r)r^m\, dr,\] 
which is given for $t<0$ and $t>0$ by 
\beq\label{E_K_neg_t}
	E_K(t;\bar r)
	=\frac{|t|^{\alpha-2+\frac{2}{\lambda}}}{2\lambda^3}
	\int_{-1}^{\frac{t}{\br^\lambda}}R(x)\Big|\frac{V(x)}{x}\Big|^2\, 
	\frac{dx}{|x|^{\alpha-1+\frac{2}{\lambda}}}
\eeq
and
\beq\label{E_K_pos_t}
	E_K(t;\bar r)
	= \frac{t^{\alpha-2+\frac{2}{\lambda}}}{2\lambda^3}
	\Big[\int_{\frac{t}{\br^\lambda}}^B+\int_B^{\infty}\Big]R(x)\Big|\frac{V(x)}{x}\Big|^2\, 
	\frac{dx}{x^{\alpha-1+\frac{2}{\lambda}}}, 
\eeq
respectively. Global boundedness of $R(x)$ and $V(x)/x$, together with assumption 
\eq{max_lam}, imply that $t\mapsto E_K(t;\br)$ is finite 
and continuous whenever $t\neq 0$. 
Evaluating at time $t=0$ yields, thanks to \eq{max_lam},
\[E_K(0;\br)=\frac{1}{2\lambda^2}R(0)\ell^2\frac{\br^{n+2-2\lambda}}{n+2-2\lambda}.\]
As the second term on  the right-hand side of \eq{E_K_pos_t} is of order 
$t^{\alpha-2+\frac{2}{\lambda}}$, and thus vanishes when $t\downarrow 0$ 
(by \eq{max_lam}), the continuity of $E_K(t;\br)$ from above at $t=0$ 
follows once it is established that 
\[\frac{t^{\alpha-2+\frac{2}{\lambda}}}{2\lambda^3}
	\int_{\frac{t}{\br^\lambda}}^B R(x)\left|\frac{V(x)}{x}\right|^2\, 
	\frac{dx}{x^{\alpha-1+\frac{2}{\lambda}}}
	\to E_K(0;\bar r)\qquad\text{as $t\downarrow 0$.}\]
Again, this follows by continuity of $R(x)|V(x)/x|^2$ at $x=0$ 
and L'H\^opital's rule. The same argument applied to \eq{E_K_neg_t}
shows that $E_K(t;\bar r)$ tends to the same limit as $t\uparrow0$.
This shows that the map $t\mapsto E_K(t;\bar r)$ is continuous 
at all times.

Finally, consider the potential energy:
\[E_P(t;\br)=\int_0^{\br} \rho(t,r)e(t,r)r^m\, dr=\frac{1}{\gamma(\gamma-1)}\int_0^{\br} \rho(t,r)c^2(t,r)r^m\, dr,\] 
which is given for $t<0$ and $t>0$ by 
\begin{align}
	E_P(t;\bar r)
	=\frac{|t|^{\alpha-2+\frac{2}{\lambda}}}{\lambda^3\gamma(\gamma-1)}
	\int_{-1}^{\frac{t}{\br^\lambda}}R(x)\Big|\frac{C(x)}{x}\Big|^2\, 
	\frac{dx}{|x|^{\alpha-1+\frac{2}{\lambda}}}\label{E_P_neg_t}
\end{align}
and
\begin{align}
	E_P(t;\bar r)
	= \frac{t^{\alpha-2+\frac{2}{\lambda}}}{\lambda^3\gamma(\gamma-1)}
	\Big[\int_{\frac{t}{\br^\lambda}}^B+\int_B^{\infty}\Big]R(x)\Big|\frac{C(x)}{x}\Big|^2\, 
	\frac{dx}{x^{\alpha-1+\frac{2}{\lambda}}},\label{E_P_pos_t} 
\end{align}
respectively. Global boundedness of $R(x)$ and $C(x)/x$, together with assumption 
\eq{max_lam}, imply that $t\mapsto E_P(t;\bar r)$ is finite and continuous at all times $t\neq 0$. 
Evaluating at time $t=0$ yields, thanks to \eq{max_lam},
\[E_P(0;\br)=\frac{1}{\lambda^2\gamma(\gamma-1)}R(0)L^2\frac{\br^{n+2-2\lambda}}{n+2-2\lambda},\]
As the second term on the right-hand side of \eq{E_P_pos_t} is of order 
$t^{\alpha-2+\frac{2}{\lambda}}$(by \eq{max_lam}), the continuity of $E_P(t;\br)$ from above at $t=0$ 
follows once it is established that 
\[\frac{t^{\alpha-2+\frac{2}{\lambda}}}{\lambda^3\gamma(\gamma-1)}
\int_{\frac{t}{\br^\lambda}}^BR(x)\left|\frac{C(x)}{x}\right|^2\, 
\frac{dx}{x^{\alpha-1+\frac{2}{\lambda}}}\to E_P(0;\bar r)
\qquad\text{as $t\downarrow 0$.}\]
As above this follows by L'H\^opital's rule and the continuity of 
$R(x)|C(x)/x|^2$ at $x=0$. Finally, the continuity of $E_P(t,\br)$ 
from below at time $t=0$ is established in the same manner.

This concludes the verification of part (ii) of Definition \ref{rad_symm_weak_soln}.

%%%%%%%%%%%%%%%%%%%%%%%%%%%%%%%%%%
\subsubsection{Local space-time integrability}
%%%%%%%%%%%%%%%%%%%%%%%%%%%%%%%%%%
Next, for part (iii) of Definition \ref{rad_symm_weak_soln}, we need to verify the 
local integrability in time and space of the functions $\rho u^2$, $p$, and
$\big[\rho \big(e+\textstyle\frac{u^2}{2}\big)+p\big]u$. Recall that we consider 
an ideal gas \eq{perf}, and that the incoming and outgoing shocks propagate 
along $x=-1$ and $x=B$, respectively.  As a consequence, to verify part (iii) 
it suffices to show that, for any fixed $\bar r>0$, the space-time integrals
\[I_\beta(\bar r):=\int_{-\bar r^\lambda}^{B\bar r^\lambda}\int_0^{\bar r} 
\rho|u|^\beta r^m\, drdt,\qquad\text{for $\beta=2,\, 3$,}\]
and
\[P_\beta(\bar r):=\int_{-\bar r^\lambda}^{B\bar r^\lambda}\int_0^{\bar r} 
p|u|^\beta r^m\, drdt,\qquad\text{for $\beta=0,\, 1$,}\]
are finite. Transforming to $dxdt$-integrals, and recalling that the fluid is at 
rest on the inside of the incoming shock, we have
\begin{align}
	I_\beta(\bar r)&=\frac{1}{\lambda^{\beta+1}}
	\left\{\int_{-1}^B\frac{R(x)|V(x)|^\beta}{|x|^{\alpha+1+\frac{\beta}{\lambda}}}
	\Big[\int_0^{|x|\bar r^\lambda}t^{\alpha+\beta\left(\frac{1}{\lambda}-1\right)}\, dt\Big]\, dx\right.\nn\\
	&\qquad\qquad\qquad\left.+\Big[\int_{B}^\infty\frac{R(x)|V(x)|^\beta}{|x|^{\alpha+1+\frac{\beta}{\lambda}}}\, dx\Big]
	\Big[\int_0^{B\bar r^\lambda}t^{\alpha+\beta\left(\frac{1}{\lambda}-1\right)}\, dt\Big]\right\}\nn\\
	&=\frac{1}{\lambda^{\beta+1}}\frac{\bar r^{\lambda(\alpha+1)+\beta(1-\lambda)}}{(\alpha+1)+\beta\left(\frac{1}{\lambda}-1\right)}
	\left\{\int_{-1}^B\frac{R(x)|V(x)|^\beta}{|x|^\beta}\, dx+B^{\alpha+1+\beta\left(\frac{1}{\lambda}-1\right)}
	\int_{B}^\infty\frac{R(x)|V(x)|^\beta}{x^{\alpha+1+\frac{\beta}{\lambda}}}\, dx \right\}.\label{I_beta}
\end{align}
Here we have used that the $dt$-integrals are finite since, for all values of 
$\lambda$, $n$, and $\beta$ under consideration, \eq{max_lam} yields
\[\alpha+\beta\Big(\frac{1}{\lambda}-1\Big)>-1.\]
As $R(x)$, $V(x)/x$, and $V(x)$ are all globally bounded, it follows from \eq{I_beta} 
that $I_\beta(\bar r)<\infty$ for any value of $\bar r$ and $\beta=2$ or $3$.

A similar computation for $P_\beta(\bar r)$ (now using that the pressure $p$ 
vanishes on the inside of the incoming shock), yields 
\begin{align}
	P_\beta(\bar r)&=\frac{1}{\gamma\lambda^{\beta+3}}
	\left\{\int_{-1}^B R(x)\Big|\frac{C(x)}{x}\Big|^2\Big|\frac{V(x)}{x}\Big|^\beta
	\frac{1}{|x|^{\alpha+1+(2+\beta)(\frac{1}{\lambda}-1)}}
	\Big[\int_0^{|x|\bar r^\lambda}t^{\alpha+(2+\beta)\left(\frac{1}{\lambda}-1\right)}\, dt\Big]\, dx\right.\nn\\
	&\qquad\qquad\qquad\left.+\Big[\int_{B}^\infty R(x)\Big|\frac{C(x)}{x}\Big|^2\Big|\frac{V(x)}{x}\Big|^\beta
	\frac{dx}{|x|^{\alpha+1+(2+\beta)(\frac{1}{\lambda}-1)}}\Big]
	\Big[\int_0^{B\bar r^\lambda}t^{\alpha+(2+\beta)\left(\frac{1}{\lambda}-1\right)}\, dt\Big] \right\}\nn\\
	&=\frac{1}{\gamma\lambda^{\beta+3}}
	\frac{\bar r^{\lambda(\alpha+1)+(2+\beta)(1-\lambda)}}{(\alpha+1)+(2+\beta)\left(\frac{1}{\lambda}-1\right)}
	\left\{\int_{-1}^BR(x)\Big|\frac{C(x)}{x}\Big|^2\Big|\frac{V(x)}{x}\Big|^\beta\, dx\right.\nn\\
	&\qquad\qquad\qquad \left.+B^{\alpha+1+(2+\beta)\left(\frac{1}{\lambda}-1\right)}
	\int_{B}^\infty R(x)\Big|\frac{C(x)}{x}\Big|^2\Big|\frac{V(x)}{x}\Big|^\beta
	\frac{dx}{|x|^{\alpha+1+(2+\beta)(\frac{1}{\lambda}-1)}} \right\}.\label{P_beta}
\end{align}
Here we have used that the $dt$-integrals are finite since, for all values of 
$\lambda$, $n$, and $\beta$ under consideration, \eq{max_lam} yields
\[\alpha+(2+\beta)\left(\frac{1}{\lambda}-1\right)>-1.\]
By global boundedness of $R(x)$, $V(x)$, and $C(x)/x$, and by \eq{max_lam},
both integrals on the right-hand side of \eq{P_beta} are finite for both $\beta=0$ and $\beta=1$. 

This concludes the verification of part (iii) of Definition \ref{rad_symm_weak_soln},
under the constraint \eq{max_lam}.

%%%%%%%%%%%%%%%%%%%%%%%%%%%%%%%%%%%
\subsection{Weak form of the equations}
%%%%%%%%%%%%%%%%%%%%%%%%%%%%%%%%%%%
Finally, for part (iv) of Definition \ref{rad_symm_weak_soln}, we need to verify the
weak forms \eq{radial_mass_weak}, \eq{radial_mom_weak}, \eq{radial_energy_weak}
of the radial equations. This requires some care since the 
solutions under consideration are unbounded at the origin. To handle this 
we shall exploit that the local integrability properties in parts (ii) and (iii) of Definition 
\ref{rad_symm_weak_soln} have been verified under the condition \eq{max_lam}. 
The issue then reduces to estimating the fluxes of the conserved quantities
across spheres of vanishing radii. 
%We shall see than the constraint \eq{lam_cond} 
%is sufficient for the energy flux across vanishing spheres to tend to zero. 

%%%%%%%%%%%%%%%%%%%%%%%%%%%%%%%%%%
\subsubsection{Weak form of the mass equation}\label{weak_forms}
%%%%%%%%%%%%%%%%%%%%%%%%%%%%%%%%%%
For a fixed $\psi\in C^1_c(\RR\times\RR^+_0)$, with $\supp\psi\subset[-T,T]\times [0,A]$,
and for any $\delta>0$, we have
\begin{align}
	M(\psi)&:=\int_{\RR}\int_{\RR^+} \left(\rho\psi_t+\rho u\psi_r\right)\, r^mdrdt 
	= \left\{\int_{\RR}\int_0^\delta +\iint_{I_\delta} +\iint_{I\!I_\delta}+\iint_{I\!I\!I_\delta}\right\}
	\left(\rho\psi_t+\rho u\psi_r\right)\, r^mdrdt \nn\\
	&=:M_\delta(\psi)
	+\left\{\iint_{I_\delta} +\iint_{I\!I_\delta}+\iint_{I\!I\!I_\delta}\right\}
	\left(\rho\psi_t+\rho u\psi_r\right)\, r^mdrdt
\end{align}
where the (open) regions $I_\delta$, $I\!I_\delta$, and $I\!I\!I_\delta$ are indicated in Figure 4
(e.g., $I_\delta$ is bounded below by $\{t=-T\}$, on the left by $\{r=\delta\}$, and on the right by
the incoming shock path).
%%%%%%%%%%%%%%%%%%
%	FIGURE 
%%%%%%%%%%%%%%%%%%
\begin{figure}\label{Figure_4}
	\centering
	\includegraphics[width=8cm,height=9cm]{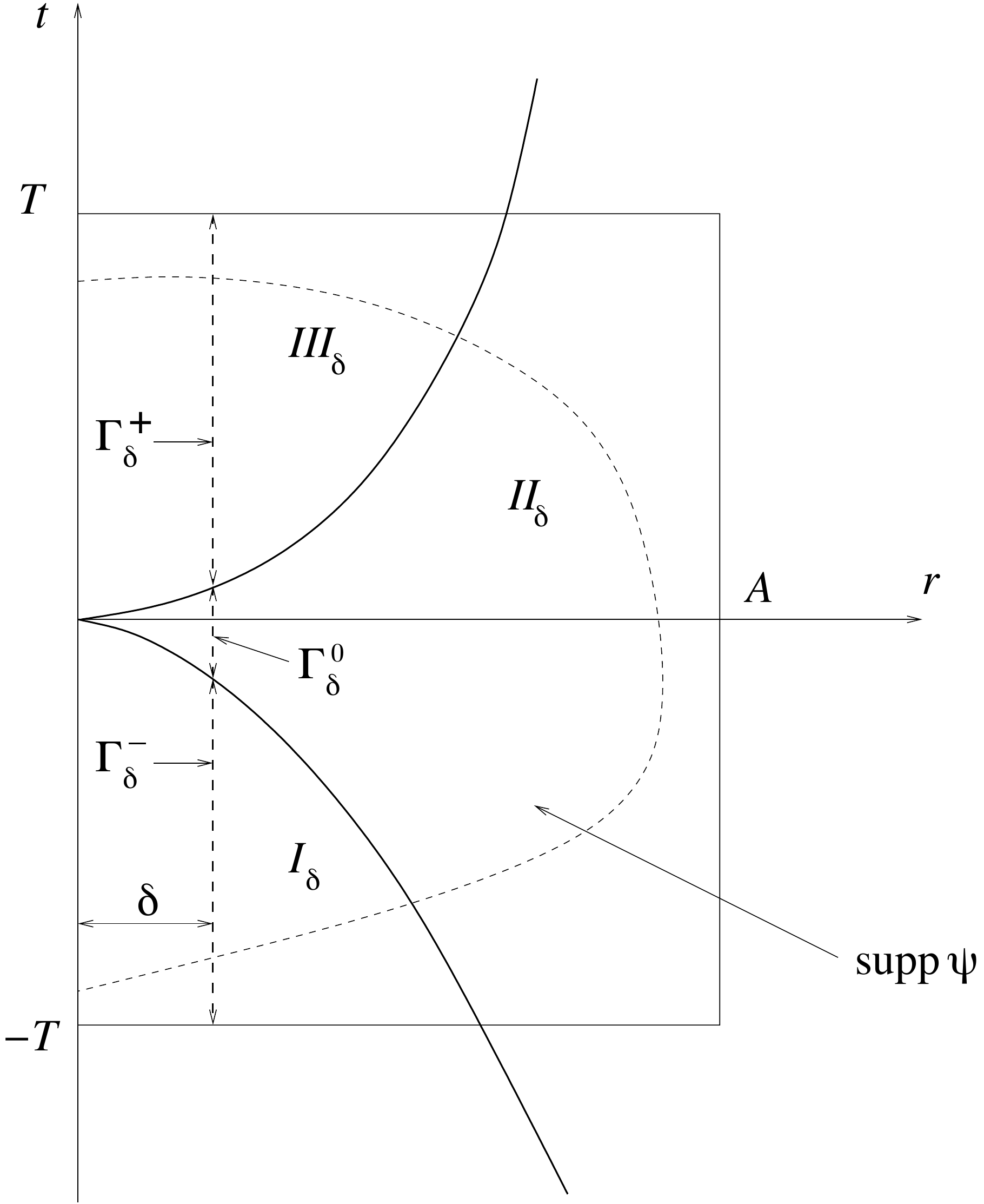}
	\caption{Regions of integration in the weak formulation.}
\end{figure}

Let ${\Gamma^-_\delta}$, ${\Gamma^0_\delta}$, and ${\Gamma^+_\delta}$ 
denote the parts of their boundaries $\partial I_\delta$, $\partial I\!I_\delta$, and $\partial I\!I\!I_\delta$, 
respectively, contained in the set $\{(t,r)\,|\,r=\delta\}$. Recall that the similarity 
shock solution is a bounded, classical solution of \eq{mass} in each of the regions 
$I_\delta$, $I\!I_\delta$, and $I\!I\!I_\delta$, and that the Rankine-Hugoniot conditions 
are satisfied\ across the incoming and outgoing shocks. Applying the divergence 
theorem therefore gives 
\[M(\psi)=M_\delta(\psi)
+\delta^m\left\{\int_{{\Gamma^0_\delta}}+\int_{{\Gamma^+_\delta}}\right\}
(\rho u\psi)(t,\delta)\,dt,\]
where we have used that $u$ vanishes along ${\Gamma^-_\delta}$.
By making the change of variables 
$t\mapsto x=t/\delta^\lambda$, we obtain
\beq\label{mass1}
	M(\psi)=M_\delta(\psi)
	-\frac{\delta^n}{\lambda}\int_{-1}^{\frac{T}{\delta^\lambda}}
	R(x)\frac{V(x)}{x}\psi(x\delta^\lambda,\delta)\,dx.
\eeq
As $R(x)$, $V(x)/x$ are globally bounded, the last term in \eq{mass1} is of order 
$\delta^{n-\lambda}$, which vanishes as $\delta\downarrow 0$ by \eq{max_lam}.
Finally, it follows from the analysis in Section \ref{cont_and_loc_integr} that both
$\rho$ and $\rho u$ belong to $L^1_{loc}(r^m drdt)$. Thus, $M_\delta(\psi)\to0$ 
as $\delta\downarrow 0$, so that $M(\psi)=0$ for each $\psi\in C^1_c(\RR\times\RR^+_0)$.
This shows that the weak form \eq{radial_mass_weak} of the radial mass equation
is satisfied.

%%%%%%%%%%%%%%%%%%%%%%%%%%%%%%%%%%
\subsubsection{Weak form of the momentum equation}
%%%%%%%%%%%%%%%%%%%%%%%%%%%%%%%%%%
For a fixed $\psi\in C^1_0(\RR\times\RR^+_0)$ and any $\delta>0$ we have
\begin{align}
	I(\psi)&:=\int_{\RR}\int_{\RR^+} \left(\rho u\psi_t
			+\rho u^2\psi_r+p\big(\psi_r+\textstyle\frac{m\psi}{r}\big)\right)\, r^mdrdt \nn\\
	&= \left\{\int_{\RR}\int_0^\delta +\iint_{I_\delta} +\iint_{I\!I_\delta}+\iint_{I\!I\!I_\delta}\right\}
	\left(\rho u\psi_t
			+\rho u^2\psi_r+p\big(\psi_r+\textstyle\frac{m\psi}{r}\big)\right)\,r^mdrdt \nn\\
	&=:I_\delta(\psi)
	+\left\{\iint_{I_\delta}+\iint_{I\!I_\delta}+\iint_{I\!I\!I_\delta}\right\}
	\left(\rho u\psi_t
			+\rho u^2\psi_r+p\big(\psi_r+\textstyle\frac{m\psi}{r}\big)\right)\, r^mdrdt.
\end{align}
Arguing as above and applying the divergence theorem gives ($x=t/\delta^\lambda$)
\begin{align}
	I(\psi)&=I_\delta(\psi)
	+\delta^m\left\{\int_{{\Gamma^0_\delta}}+\int_{{\Gamma^+_\delta}}\right\}
	((\rho u^2+p)\psi)(t,\delta)\,dt\nn\\
	&= I_\delta(\psi)
	+\frac{\delta^{n+1-\lambda}}{\lambda^2}\int_{-1}^{\frac{T}{\delta^\lambda}}
	R(x)\Big[\Big|\frac{V(x)}{x}\Big|^2+\frac{1}{\gamma}\Big|\frac{C(x)}{x}\Big|^2\Big]
	\psi(x\delta^\lambda,\delta)\,dx,\label{I_psi}
\end{align}
where we have used that $u$ and $p$ both vanish along ${\Gamma^-_\delta}$.
Recalling the observation in Remark \ref{psi_0_rmk}, and using global boundedness
of $R(x)(V(x)/x)^2$ and $R(x)(C(x)/x)^2$, we obtain that 
\[\delta^{n+1-\lambda}\int_{-1}^{\frac{T}{\delta^\lambda}}
	R(x)\Big[\Big|\frac{V(x)}{x}\Big|^2+\frac{1}{\gamma}\Big|\frac{C(x)}{x}\Big|^2\Big]
	\psi(x\delta^\lambda,\delta)\,dx\lesssim \delta^{n+2-2\lambda},\]
which tends to zero as $\delta\downarrow 0$ by \eq{max_lam}. Finally, to show that $I_\delta(\psi)$ also 
vanishes with $\delta$ we first use Remark \ref{psi_0_rmk} to bound the 
function $\frac{\psi}{r}$ by a constant, and then use that, according to the analysis above, 
the quantities $\rho u$, $\rho u^2$, and $p$ all belong to $L^1_{loc}(r^m drdt)$.
This shows that also $I_\delta(\psi)\to0$ as $\delta\downarrow 0$. Thus,
$I(\psi)=0$ for each $\psi\in C^1_0(\RR\times\RR^+_0)$,
showing that the weak form \eq{radial_mom_weak} of the momentum equation
is satisfied.

%%%%%%%%%%%%%%%%%%%%%%%%%%%%%%%%%%
\subsubsection{Weak form of the energy equation}
%%%%%%%%%%%%%%%%%%%%%%%%%%%%%%%%%%
For a fixed $\psi\in C^1_c(\RR\times\RR^+_0)$ and any $\delta>0$ we have
\begin{align}
	E(\psi)&:=\int_\RR\int_{\RR^+} \left(\rho \big(e+\textstyle\frac{u^2}{2}\big) \psi_t
			+\left[\rho \big(e+\textstyle\frac{u^2}{2}\big)+p\right] u\psi_r\right)\, r^mdrdt \nn\\
	&= \left\{\int_{\RR}\int_0^\delta +\iint_{I_\delta} +\iint_{I\!I_\delta}+\iint_{I\!I\!I_\delta}\right\}
	\left(\rho \big(e+\textstyle\frac{u^2}{2}\big) \psi_t
			+\left[\rho \big(e+\textstyle\frac{u^2}{2}\big)+p\right] u\psi_r\right)\, r^mdrdt \nn\\
	&=:E_\delta(\psi)
	+\left\{\iint_{I_\delta}+\iint_{I\!I_\delta}+\iint_{I\!I\!I_\delta}\right\}
	 \left(\rho \big(e+\textstyle\frac{u^2}{2}\big) \psi_t
			+\left[\rho \big(e+\textstyle\frac{u^2}{2}\big)+p\right] u\psi_r\right)\, r^mdrdt .
\end{align}
Arguing as above and applying the divergence theorem gives ($x=t/\delta^\lambda$)
\begin{align}
	E(\psi)&=E_\delta(\psi)
	+\delta^m\left\{\int_{{\Gamma^0_\delta}}+\int_{{\Gamma^+_\delta}}\right\}
	\left[\rho u\big(e+\frac{1}{2}u^2+\frac{p}{\rho}\big)\psi\right](t,\delta)\,dt\nn\\
	&= E_\delta(\psi)
	+\frac{\delta^{n+2-2\lambda}}{\lambda^3}\int_{-1}^{\frac{T}{\delta^\lambda}}
	R(x)\frac{V(x)}{x}\left(\frac{1}{2}\left|\frac{V(x)}{x}\right|^2+\frac{1}{\gamma-1}\left|\frac{C(x)}{x}\right|^2\right)
	\psi(x\delta^\lambda,\delta)\,dx,\label{E_psi}
\end{align}
where we have used that $u$ vanishes along ${\Gamma^-_\delta}$.
Recalling the global boundedness
of $R(x)$, $V(x)$, $V(x)/x$, and $R(x)(C(x)/x)^2$, as well as the bound 
\eq{aux}, we obtain that the last integral in \eq{E_psi} is bounded by 
\[\lesssim 1+\int_{1}^{\frac{T}{\delta^\lambda}} x^{-3}
+x^{-2(1-\frac{1}{\lambda})-1}\,dx\lesssim1+\delta^{2\lambda}
	+\delta^{2(\lambda-1)}\qquad\text{as $\delta\downarrow 0$.}\]
According to \eq{max_lam} we therefore have that the last term on the 
right-hand side of \eq{E_psi} vanishes as $\delta\downarrow 0$.
Finally, under the same constraint on $\lambda$, the argument in Section 
\ref{cont_and_loc_integr} showed that the quantities $\rho e\propto p$, 
$\rho u^2$, $\rho ue\propto up$, and $\rho u^3$, all
belong to $L^1_{loc}(r^mdrdt)$. In particular, it follows that $E_\delta(\psi)$  
vanishes as $\delta\downarrow 0$. Thus, $E(\psi)=0$ for each $\psi\in C^1_c(\RR\times\RR^+_0)$,
showing that the weak form \eq{radial_energy_weak} of the energy equation
is satisfied.

This concludes the proof of Theorem \ref{main_thm}.

\bigskip
\paragraph{\bf Acknowledgment:}
This work was supported in part by NSF awards DMS-1311353 (Jenssen) 
and DMS-1714912 (Tsikkou).

\end{document}